\documentclass[a4paper,10pt]{amsart}
\usepackage[displaymath, mathlines]{lineno}

\usepackage{amsmath,amssymb,amsthm,xcolor,graphicx,subfigure, hyperref}
\usepackage{tikz}

\newtheorem{thm}{Theorem}[section]

\newtheorem{lemma}[thm]{Lemma}
\newtheorem{cor}{Corollary}[section]

\newcommand{\maps}{\rightarrow}
\newcommand{\intersect}{\cap}

\newcommand{\be}{\beta}

\newcommand{\del}{\delta}
\newcommand{\ep}{\epsilon}

\newcommand{\lam}{\lambda}

\newcommand{\x}{{\bf x}}

\newcommand{\R}{\mathbb{R}}

\newcommand{\beq}{\begin{equation}}
\newcommand{\eeq}{\end{equation}}

\numberwithin{equation}{section}

\newcommand \mo{\mathcal{O}}
\newcommand \e{\varepsilon}
\newcommand \mathp{\mathcal{P}}

\title[]{Lower bounds for the truncated\\ hilbert transform}
\author{Rima Alaifari}
\address[Rima Alaifari]{ Department of Mathematics, ETH Z\"{u}rich, R\"{a}mistrasse 101, 8092 Z\"{u}rich, Switzerland}
\email{rima.alaifari@math.ethz.ch}

\author{Lillian B. Pierce}
\address[Lillian B. Pierce]{Department of Mathematics, Duke University, 120 Science Drive, Durham NC 27708, and Hausdorff Center for Mathematics, Endenicher Allee 62, Bonn, Germany}
\email{pierce@math.duke.edu and pierce@math.uni-bonn.de}

\author{Stefan Steinerberger}
\address[Stefan Steinerberger]{Department of Mathematics, Yale University, 10 Hillhouse Avenue, 06511 CT, USA}
\email{stefan.steinerberger@yale.edu}

\begin{document}
\begin{abstract}
Given two intervals $I, J \subset \mathbb{R}$,
 we ask whether it is possible to reconstruct a 
real-valued function
$f \in L^2(I)$ from knowing its Hilbert transform $Hf$ on $J$. 
When neither interval is fully contained in the other, this problem has a unique answer (the nullspace is trivial) but is severely ill-posed. We isolate the difficulty and show that by restricting $f$ to functions with controlled total variation, reconstruction becomes stable. In particular, for functions $f \in H^1(I)$, we show that
$$  \|Hf\|_{L^2(J)} \geq c_1 \exp{\left(-c_2 \frac{\|f_x\|_{L^2(I)}}{\|f\|_{L^2(I)}}\right)} \| f \|_{L^2(I)} ,$$
for some constants $c_1, c_2 > 0$ depending only on $I, J$. This inequality is sharp, but we conjecture that
$\|f_x\|_{L^2(I)}$ can be replaced by $\|f_x\|_{L^1(I)}$.
\end{abstract}
\maketitle

\section{Introduction and Motivation}
\subsection{Hilbert transform.} The Hilbert transform $H:L^2(\mathbb{R}) \rightarrow L^2(\mathbb{R})$ is a well-studied unitary operator given by
$$ (Hf)(x) = \frac{1}{\pi}\mbox{p.v.}\int_{\mathbb{R}}{\frac{f(y)}{x-y}dy},$$
where p.v. indicates that the integral is to be understood as a principal value. On $L^2(\mathbb{R})$ it can
alternatively be defined via the Fourier multiplier $-i \mbox{sgn}(\xi)$. The Hilbert transform appears naturally 
in many different settings in pure and applied mathematics. In particular, it plays an important role in the mathematical
study of inverse problems arising in medical imaging (see \S \ref{sec_medim}), which motivates the following fundamental question:\\

\begin{quote}
\textbf{Inversion problem.} Given two finite intervals $I, J \subset \mathbb{R}$ and a real-valued function $f \in L^2(I)$, 
when can $f$ be reconstructed from knowing $Hf$ on $J$? \\
\end{quote}

\begin{center}
\begin{figure}[h!]
\begin{tikzpicture}[xscale=7,yscale=1.5]
\draw [domain=0:1, samples = 300] plot (\x, {-0.15269*sin(2*pi*\x r)    +    0.4830*sin(3*pi*\x r)  +  0.3084*sin(4*pi*\x r)  + 0.80509*sin(5*pi*\x r)}  );
\draw [domain=0:1] plot (\x, {0}  );
\end{tikzpicture}
\caption{A function $f$ on $[0,1]$ with $\|Hf\|_{L^2([2,3])} \sim 10^{-7}\|f\|_{L^2([0,1])}$.} 
\end{figure}
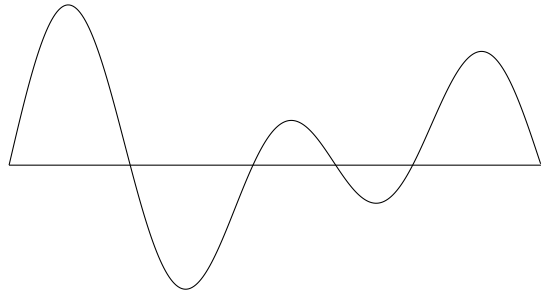
\end{center}

For illustration, let us consider first the particular case in which the intervals $I$ and $J$ are disjoint.
The Hilbert transform is an integral operator --- thus if a function $f$ is compactly supported and we consider the Hilbert transform
$Hf$ only outside of that support, the singularity of the kernel never plays a role and the operator is compact (i.e. smoothing). It
is clear from basic principles in functional analysis that the inversion of a compact operator will not yield a bounded operator.

\subsection{The phenomenon in practice} Let us understand just how ill-posed the problem actually is. Consider the Hilbert transform applied to functions with support on $I = [0,1]$ and then
take its restriction on $J=[2,3]$, leading to the truncated Hilbert transform
$$ H_T = \chi_{[2,3]}H(\chi_{[0,1]}f).$$
The operator $H_T:L^2(\mathbb{R}) \rightarrow L^2(\mathbb{R})$ is a compact integral operator.
As a consequence of it being compact,  for every $\varepsilon > 0$
we can find a function $f \in L^2([0,1])$ with
$$ \|H_T f\|_{L^2([2,3])} \leq \varepsilon \| f\|_{L^2([0,1])}.$$
Such $f$ are actually very easy to find: let $\phi_1, \phi_2, \dots, \phi_n$ 
denote any $n$ orthonormal real-valued functions in $L^2([0,1])$ and consider the subspace 
$$ \mathcal{S} =  \mbox{span}(\phi_1, \phi_2, \dots, \phi_n)$$ 
spanned by these functions. Then, putting
$$ f = \sum_{k=1}^{n}{a_k \phi_k}$$
for real coefficients $a_k$
immediately implies that
\begin{align*}
\|H_T f\|_{L^2([2,3])}^2 &= \int_{2}^{3}{ \left(H_T\sum_{k=1}^{n}{a_k \phi_k} \right)^2 dx} \\
&= \int_{2}^{3}{ \sum_{k,l = 1}^{n}{a_k a_l (H_T\phi_k) (H_T \phi_l)}dx} \\
&=  \sum_{k,l = 1}^{n}{a_k a_l \int_{2}^{3}{  (H_T\phi_k) (H_T \phi_l)dx}}.
\end{align*}
This yields that the symmetric $n \times n$ matrix
$$ A = \left( \int_{2}^{3}{  (H_T\phi_k) (H_T \phi_l)dx} \right)_{k,l = 1}^{n} $$
will satisfy the relation
$$ \inf_{f \in \mathcal{S}}\frac{ \|H_T f\|_{L^2([2,3])} }{ \| f\|_{L^2([0,1])}} = \sqrt{\lambda_{\mathrm{min}}(A)},$$
where $\lam_{\mathrm{min}}(A)$ denotes the least eigenvalue of $A$.
Put differently, $H_T: \mathcal{S} \rightarrow H_T(\mathcal{S})$ may be a bijection but the inverse operator is very sensitive to noise
in the measurement because
$$ \| H_T^{-1}\|_{H_T(\mathcal{S}) \rightarrow \mathcal{S}} = \frac{1}{\sqrt{\lambda_{\mathrm{min}}(A)}} \qquad \mbox{will be very big.}$$
So far, this discussion could apply to a wide class of operators; focusing on our situation, the key point is that the spectrum of $H_T$ is rapidly decaying and therefore $\lambda_{\mathrm{min}}(A)$
will always be very small, independent of the $n$ orthonormal functions we pick; indeed (see equation (\ref{signupbd})), there exist real constants $C, \be>0$ dependent only on the intervals $I,J$ such that
$$\lambda_{\mathrm{min}}(A) \leq C e^{-\beta n}.$$

\subsection{An explicit example.} We consider a numerical example; let $I = [0,1]$, $J=[2,3]$
and consider the subspace spanned by
$$\phi_i = \sqrt{3}\chi_{[\frac{i-1}{3},\frac{i}{3}]} \qquad \mbox{for}~1 \leq i \leq 3.$$
The choice of these functions is solely motivated by the fact that $H_T \phi_i$ can be written down in
closed form, which simplifies computation. Then the matrix $A$ has the eigenvalues 
$$ \sim \left\{0.28,~ 0.00013, ~2.2 \cdot 10^{-8} \right\} \qquad \mbox{which are very rapidly decaying.}$$
As a consequence, there exists a step function $g$, which is constant on the three intervals of length $1/3$ in $[0,1]$
(and is therefore certainly quite simple) but nonetheless satisfies
$$ \|H_T g\|_{L^2([2,3])} \leq 10^{-7} \| g\|_{L^2([0,1])}.$$
It is interesting to compare this with a larger subspace. Pick now, for comparison,
$$\phi_i = \sqrt{5}\chi_{[\frac{i-1}{5},\frac{i}{5}]} \qquad \mbox{for}~1 \leq i \leq 5.$$
The smallest eigenvector of the arising matrix is $\lambda_{\mathrm{min}}(A) \leq 10^{-15},$ allowing for the
construction of a step function $h$ with
$$ \|H_T h \|_{L^2([2,3])} \leq 10^{-15} \| h\|_{L^2([0,1])}.$$
\begin{figure}[h!, scale=0.8]
\begin{minipage}[b]{0.48\linewidth}
\centering
\includegraphics[width=\textwidth]{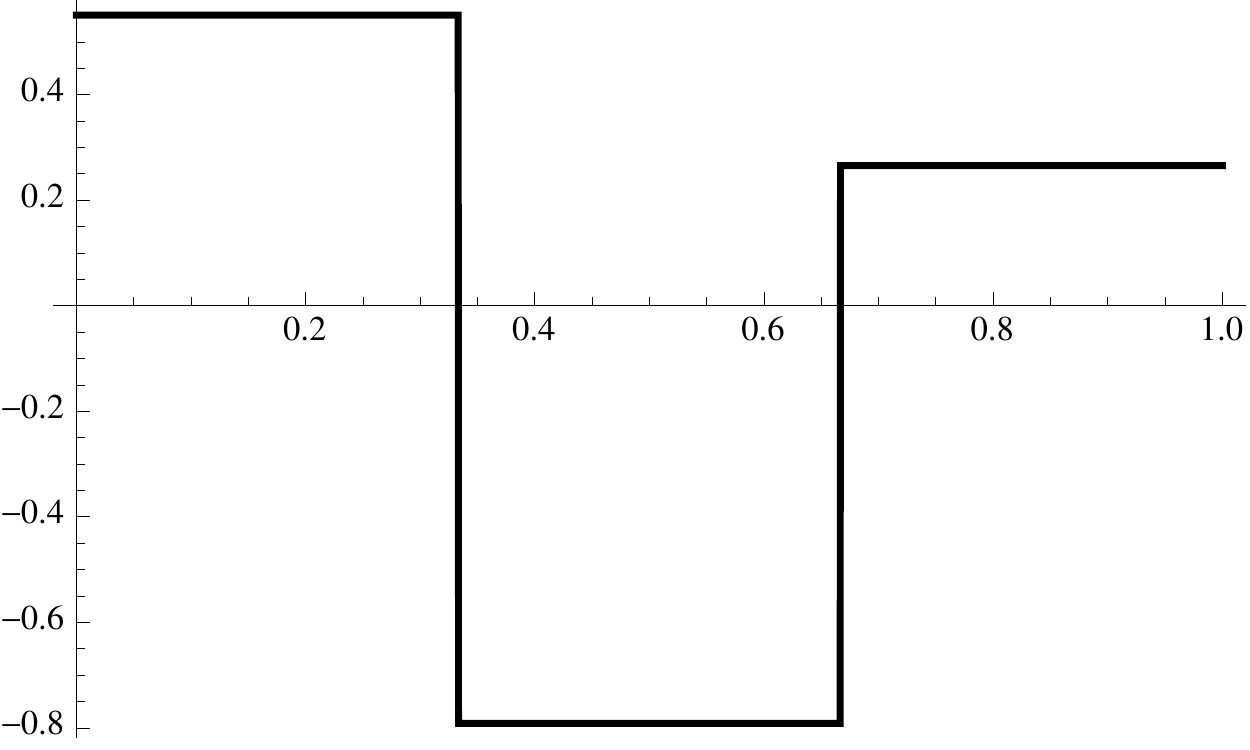}
\caption{\newline The function $g$ ($n=3$).}
\end{minipage}
\begin{minipage}[b]{0.48\linewidth}
\centering
\includegraphics[width=\textwidth]{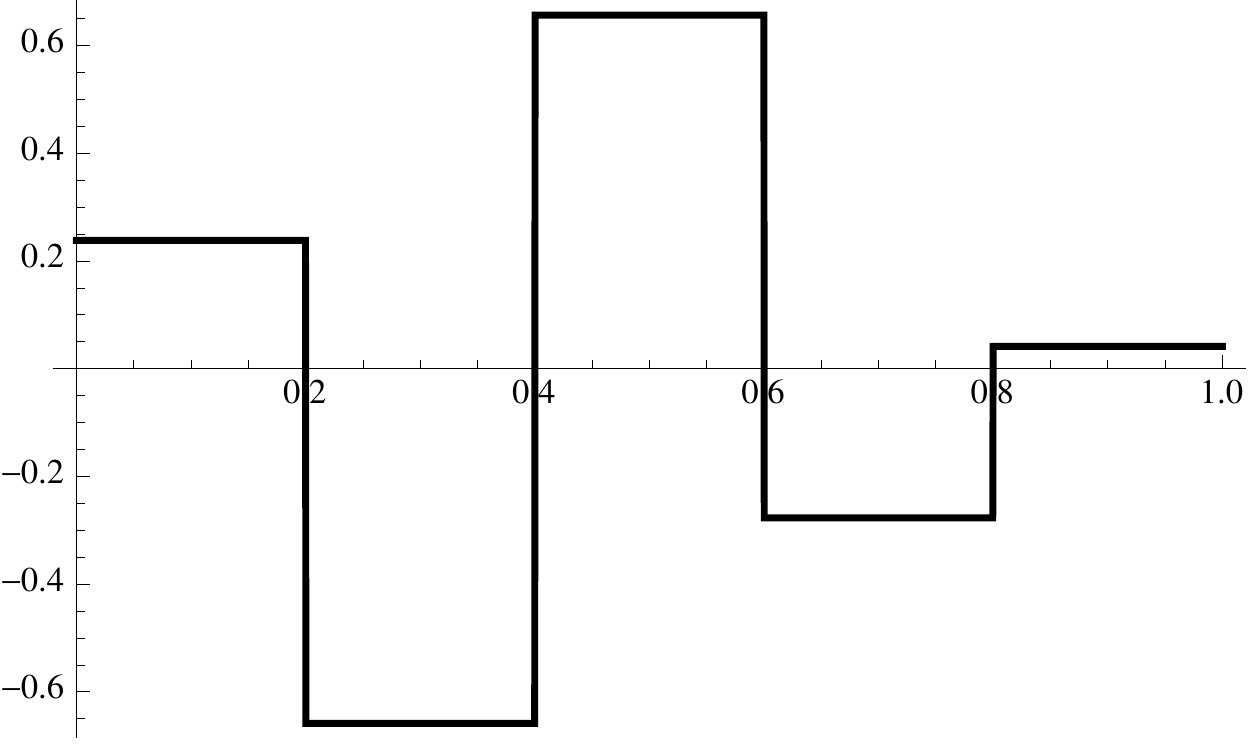}
\caption{\newline The function $h$ ($n=5$).}
\end{minipage}
\end{figure}

The contribution of our paper may now be phrased as follows: the fact that these functions are highly oscillatory
is not a coincidence; indeed, it is the purpose of this paper to point out that an
inverse inequality is true: \textit{reconstruction becomes stable for functions with controlled total variation.}
While there are a variety of techniques for understanding how to bound oscillating quantities from above (e.g. stationary phase), it is usually much 
harder to control oscillation from below -- finding sharp quantitative versions of the above statement falls precisely into this class of
problems; as such, we believe it to be very interesting. The same problem could be of great interest for more general integral
operators, where a similar phenomenon should be true generically (see Section \ref{sec_general}).

\subsection{Configurations of the intervals: four cases.} The precise nature of the problem of reconstructing a function supported on $I$ from its Hilbert transform on $J$ will depend
on the relation between $I$ and $J$. To address this question adequately, it is useful to distinguish four cases:

\begin{enumerate}
\item The Hilbert transform is known on an interval $J$ that covers the support $I$ of $f$ (that is, $I \subset J$). In this case inversion is stable (the solution operator is bounded) and an explicit inversion formula is known \cite{tri}. \\
\begin{figure}[ht!]
     \begin{center}
           \includegraphics[width=0.4\textwidth]{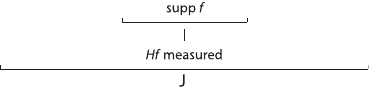}
           \end{center}
\end{figure}
\item The Hilbert transform is known only on an interval $J$ that is a subset of the support $I$ of $f$ (that is, $I \supseteq J$). In tomographic reconstruction, this case is known as the \textit{interior problem}  \cite{cour, kat2, ekat, kcnd, yyw}. \\
\begin{figure}[ht!]
     \begin{center}
           \includegraphics[width=0.4\textwidth]{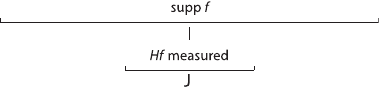}
           \end{center}
\end{figure}
\item The Hilbert transform is known only outside of the support of $f$ (that is, $I \cap J = \emptyset$). We will refer to this scenario as the truncated Hilbert transform \textit{with a gap.} The singular value decomposition of the underlying operator has been studied in \cite{kat}.\\
\begin{figure}[ht!]
     \begin{center}
           \includegraphics[width=0.4\textwidth]{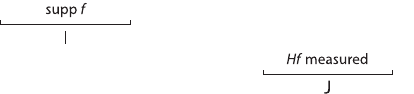}
           \end{center}
\end{figure}
\item If none of the above is the case and the Hilbert transform is known on an interval $J$ that overlaps with the support $I$ of $f,$ we call this the truncated Hilbert transform \textit{with overlap.} For this case a pointwise stability estimate has been shown in \cite{defrise}. The spectral properties of the underlying operator are the subject of \cite{reema1, reema2}.\\
\begin{figure}[ht!]
     \begin{center}
           \includegraphics[width=0.4\textwidth]{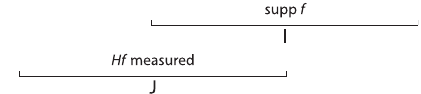}
           \end{center}
\end{figure}
\end{enumerate}

 In this paper, we consider Cases 3 and 4. For these, $f$ is supported on $I$ and $Hf$ is known on $J$, where $I$ and $J$ are non-empty finite intervals on $\mathbb{R}$, such that $I \not\supseteq J$ and $I \not\subset J.$ Let $\mathcal{P}_{\Omega}$ stand for the projection operator onto a set $\Omega \subset \R$: 
 \[ \text{$(\mathcal{P}_{\Omega} f) (x) = f(x)$ if $x \in \Omega$, $(\mathcal{P}_{\Omega} f)(x) = 0$ otherwise.}\]
  We will use the notation $H_T=\mathp_J H \mathp_I$ to denote the truncated Hilbert transform (\textit{with a gap} or \textit{with overlap}), specialized to the intervals $I$ and $J$.

\subsection{Applications in medical imaging.}\label{sec_medim}
The problem of reconstructing a function from its partially known Hilbert transform arises naturally in computerized tomography: assume a 2D or 3D object is 
illuminated from various directions by a penetrating beam (usually X-rays) and that the attenuation of the 
X-ray signals is measured by a set of detectors. Then, one seeks to reconstruct the object from the measured 
attenuation, which can be modeled as the Radon transform data of the object. If the directions along which 
the Radon transform is measured are sufficiently dense, the problem and its
solution are well-understood (cf. \cite{nat}). When the directions are not sufficiently dense the problem 
is more complicated. One such setting is the case of \textit{truncated projections} and occurs when only 
a sub-region of the object is illuminated by a sufficiently dense set of directions. Going back to a result by Gelfand \& Graev \cite{gel}, the method
of differentiated back-projection allows one to reduce the problem to solving a family of one-dimensional problems which 
consist of inverting the Hilbert transform data on a finite segment of the line. If one knew $Hf$ on all of $\mathbb{R}$, 
this would be trivial, since $H^{-1} = -H$. \\

In practice, $Hf$ is measured on only a finite segment, giving rise to the different configurations 1 through 4 and the resulting reconstruction problems. In this paper, we focus on Case 3 (the truncated Hilbert transform 
with a gap) and Case 4 (the truncated Hilbert transform 
with overlap), which are the most unstable from the point of view of functional analysis.
In fact, both these cases are \textit{severely} ill-posed, meaning that the singular values of the underlying operator decay to zero at 
an exponential rate. (For the asymptotic analysis of the singular value decomposition in Case 3 we refer to Katsevich \& Tovbis \cite{kat3}; for Case 4, see \cite{reema2}.)
In Case 3, the Hilbert transform is an integral operator with a smooth kernel
and is thus compact.
In general, one would expect Case 4 to be better behaved with respect to the inversion problem
as long as the functions have, say, a fixed proportion of their $L^2$--mass supported on $I \cap J$. By considering the subproblem arising in Case 4 when we consider functions with compact support bounded away from J, we see that all the difficulties of Case 3 must also be present in Case 4. Inverse estimates specifically tailored to Case 4, which show their strength precisely for functions not supported away from J, are presented in Section 2.4 \ref{sec:results-overlap}.

\subsection{Questions of regularity.} In order to situate our results in terms of the role of regularity, it is worth observing that the actual problem of reconstruction is \textit{not} easier for smooth functions. This is easily seen in Case 3: when $I$ and $J$ are disjoint, there is less stability of the inversion problem of the truncated Hilbert transform; in this case the truncated Hilbert transform turns into a highly regular smoothing integral operator (in contrast to the classical Hilbert transform which is the fundamental example of a singular integral operator). Indeed, when $I$ and $J$ are disjoint, the singularity of the Hilbert kernel never comes into play. This smoothing property of the truncated Hilbert transform with a gap allows one to approximate any function $f \in L^2(I)$
by $C^{\infty}$ functions $f_n$ such that $H_T f_n \rightarrow H_T f$ in $L^{\infty}(J)$. This can be seen from 
$$\| H_T f_n - H_T f\|_{L^{\infty}(J)} \leq \tilde{c} \| f_n - f\|_{L^1(I)} \leq c \| f_n - f\|_{L^2(I)},$$
where $$\tilde{c} = \max_{x \in I, y \in J} \frac{1}{|y-x|}$$ and $c = \tilde{c} \cdot |I|^{1/2}.$  Yet while the problem of reconstruction is in theory no easier for smooth functions, our current methods will be able to obtain improved estimates for smooth functions (whereas any argument yielding a sharp result should be oblivious to questions of regularity).
Another classical property we will make use of is that one can always approximate a function of bounded variation by smooth functions while controlling their total variation (TV). More precisely, we have the following lemma (which we prove in \S \ref{sec_lemma1}):
\begin{lemma}\label{lemma_approx}
Given a function $f \in BV(I)$ satisfying $f(x_0)=0$ for at least one $x_0 \in I$, there exists a sequence $f_n \in C^{\infty}_{c}(I)$ such that
$$\| f_n  -  f \|_{L^2(I)} \to 0 \qquad \mbox{and} \qquad \left| f_n \right|_{TV} \leq 3\cdot \left| f \right|_{TV}.$$
\end{lemma}
We note that the condition that $f$ vanishes at least at one point in the interval will not be a significant restriction in our applications of this lemma (see Lemma \ref{polydecay}, and subsequent remarks, for example).

\textbf{Notation.} In the following, $I$ and $J$ always denote finite open intervals on $\mathbb{R}.$ We write $C_c^N(I)$ for the space of $N$-times differentiable functions compactly supported on $I$. As conventional, $H^k$ denotes the Sobolev space $W^{k,2}$, and
we recall the following well-known inclusions for a finite interval $\Omega \subset \mathbb{R}:$
$$H^1(\Omega) \subset W^{1,1}(\Omega) \subset BV(\Omega) \subset L^2(\Omega) \subset L^1(\Omega).$$

\section{Statement of results}\label{sec:results}
\subsection{Functions of bounded variation.} Our first finding establishes a stability result for functions of bounded variation. This seems to be the appropriate notion to exclude strong
oscillation while still allowing for rather rough functions with jump discontinuities. The total variation (TV) model has been studied as a regularizing constraint in computerized tomography before, see e.g. \cite{sidky}.

\begin{thm}\label{thm1} Let $I, J \subset \mathbb{R}$ be intervals in the configuration of Case 3 or Case 4 and consider functions $f \in BV(I)$ supported on $I$. There exists a positive function $h:[0, \infty) \rightarrow \mathbb{R}_{+}$ (depending only on $I,J$) such that
$$ \| H f\|_{L^2(J)} \geq h\left(\frac{\left| f \right|_{\text{TV}}}{\|f\|_{L^2(I)}}\right)\|f\|_{L^2(I)},$$
where $\left| \cdot \right|_{\text{TV}}$ denotes the total variation of $f$.
\end{thm}
We conjecture  
$$ h(\kappa) \geq c_1 e^{-c_2 \kappa}$$
for constants $c_1, c_2 > 0$ depending only on $I$ and $J$. \\

The relation between Theorem \ref{thm1} and the reconstruction problem can easily be made explicit.
In the application of computerized tomography one needs to solve $H_T f = g$ for $f$,
 given a right-hand side $g$. In practice, $g$ has to be measured and is thus never known exactly, but only up to a certain 
accuracy. Since the range of the operator $H_T$ is dense but not closed in $L^2(J)$, the inversion of $H_T$ is ill-posed, see
\cite{reema1}. As a consequence, the solution $f$ to $H_T f = g$ does not depend continuously on the right-hand side. 
In particular, small perturbations in $g$ due to measurement noise might change the solution completely, making the outcome unreliable.
Given a function $g$ representing exact data, of which we know only a noisy measurement $g^{\delta}$ 
and the noise level $\| g - g^{\delta} \|_{L^2(J)} \leq \delta$, quantitative results taking the form of Theorem \ref{thm1} will enable stable reconstruction, under the assumption
that the true solution $f_{\text{ex}}$ to $H_T f = g$ has bounded variation (see Corollary \ref{cor-stab}).

\subsection{Weakly differentiable functions.} 
We now turn our focus to proving quantitative versions of Theorem \ref{thm1} for more regular functions $f$.
For weakly differentiable functions we can actually write
$$ \left| f \right|_{\text{TV}} = \int_{I}{|f_x(x)|dx}$$
and thus identify the total variation with $\|f_x\|_{L^1(I)}$. In light of Theorem \ref{thm1}, the total variation seems to be the natural quantity with which to track the behavior of regular functions, and we conjecture that
\beq\label{thm2_ineq_conj}
 \|H f\|_{L^2(J)} \geq c_1 \exp{\left(-c_2\frac{ \|f_x\|_{L^1(I)}}{\|f\|_{L^2(I)}}\right)} \| f \|_{L^2(I)} .
 \eeq
An inequality of this form would quantify the physically intuitive notion that tomographic reconstruction is more
difficult for inhomogeneous objects with high variation in density than it is for relatively uniform objects.
Our first result toward this conjecture considers
 $\|f_x\|_{L^2(I)}$ instead, which provides access to Hilbert space techniques that allow us to prove the
following statement:
\begin{thm}\label{thm2} Let $I, J \subset \mathbb{R}$ be intervals in the configuration of Case 3 or Case 4. Then, for any $f \in H^1(I)$,
\beq\label{thm2_ineq}
 \|H f\|_{L^2(J)} \geq c_1 \exp{\left(-c_2\frac{ \|f_x\|_{L^2(I)}}{\|f\|_{L^2(I)}}\right)} \| f \|_{L^2(I)} ,
 \eeq
for some constants $c_1, c_2 > 0$ depending only on $I, J$. 
\end{thm}

We note that Theorem \ref{thm2} is weaker than the conjectured inequality (\ref{thm2_ineq_conj}): a step function $f$, for example, can be approximated by smooth functions $f_n$ in such a way that $\|(f_n)_x\|_{L^1}$
remains controlled by the total variation of $f$. However, this is no longer true for $\|(f_n)_x\|_{L^2},$ which must
necessarily blow up.
Yet we may improve on Theorem \ref{thm2} if $f$ is sufficiently smooth and obtain a result which in certain cases is as strong as the conjectured relation (\ref{thm2_ineq_conj}):
\begin{thm}\label{thm2a} Let $I, J \subset \mathbb{R}$ be intervals in the configuration of Case 3 or Case 4. Then there exists an order 2 differential operator $L_I$ and for any $M \geq 1$ a dense class $A_M$ of $L^2$-functions (defined in \S \ref{sec:diff-operator}) such that for any $f \in A_M$, 
\beq\label{thm2_ineqa}
 \|H f\|_{L^2(J)} \geq c_{1,M} \exp{\left(-c_{2,M}\left(\frac{ \|(L_I^M f)_x\|_{L^2(I)}}{\|f\|_{L^2(I)}}\right)^{\frac{1}{2M+1}} \right)} \| f \|_{L^2(I)} ,
 \eeq
for some constants $c_{1,M}, c_{2,M} > 0$ depending only on $I, J$ and $M.$ As $M \maps \infty$, $c_{1,M},c_{2,M}$ tend to finite limits $c_1,c_2>0$. Furthermore, $C_c^{2M+1}(I) \subset A_M$ and $C_c^{\infty}(I) \subset \bigcap_{n=1}^{\infty} A_n.$
\end{thm}
In certain examples, this result approaches the desired conjecture (\ref{thm2_ineq_conj}). 
Consider, for instance, the interval $I=(0,1)$, use dilation to move the support of the function $f_N(x) = \sin (2\pi Nx)\chi_{[0,1]}$
strictly inside the unit interval and convolve with a compactly supported $C^{\infty}$ bump function. In this case $(L_I^M f_N)_x$
contains a main term of size $(2\pi N)^{2M+1} \cos (2\pi Nx)$. Then, morally speaking, the theorem implies
\begin{eqnarray*}
  \|H f_N\|_{L^2(J)} &\geq &c_{1,M} \exp{\left(-c_{2,M} 2\pi N \left( c_J \frac{\| \cos(2\pi Nx) \|_{L^2(I)}}{\| f_N \|_{L^2(I)}}\right)^{\frac{1}{2M+1}} \right)} \|  f_N\|_{L^2(I)} \\
  & \geq & c_1 \exp{\left(-c_2N\right)} \|  f_N\|_{L^2(I)}
  \end{eqnarray*}
as $M \maps \infty.$ Here $c_{1,M}, c_{2,M}, c_1, c_2$ are as in Theorem \ref{thm2a} and $c_J$ is a constant depending only on $J$ (since we have fixed the interval $I$). This is of the form (\ref{thm2_ineq_conj}), because in this example 
\beq\label{relation}
\| (f_N)_x \|_{L^1(I)}/\| f_N \|_{L^2(I)}  \approx N.
\eeq

\subsection{A quantitative result for functions with bounded variation.} Our next result gives a different type of result toward the conjecture (\ref{thm2_ineq_conj}), now for functions $f \in W^{1,1}(I)$, and with quadratic scaling within
the exponential.
We note that the inequality below is superior to the bound given by Theorem \ref{thm2} only for
functions with $\|f_x\|_{L^2} \gg \|f_x\|_{L^1}^2/\| f\|_{L^2(I)}$.
\begin{thm}\label{thm3} Let $I, J \subset \mathbb{R}$ be intervals in the configuration of Case 3 or Case 4. Then, for any $f \in W^{1,1}(I)$,
$$ \|H f\|_{L^2(J)} \geq c_1 \exp{\left(-c_2\frac{ |f|^2_{TV}}{\|f\|^2_{L^2(I)}}\right)} \| f \|_{L^2(I)} ,$$
for some constants $c_1, c_2 > 0$ depending only on $I, J$. 
\end{thm}

Theorem \ref{thm3} provides a stability estimate (independent of a specific algorithm) for the reconstruction 
of a solution $f$ to $H_T f = g$. 
\begin{cor}[Stable reconstruction]\label{cor-stab}
Let $g \in \text{Ran}(H_T)$, such that $$H_T f_{\text{ex}} = g,$$ and $| f_{\text{ex}}|_{TV} \leq \kappa$.  
Furthermore, let $g^{\delta} \in L^2(J)$ satisfy $$\| g - g^{\delta}\|_{L^2(J)} \leq \delta$$ for some $\delta >0$ and define the set of admissible solutions to be 
$$S(\delta, g^{\delta}) = \{ f \in W^{1,1}(I): \| H_T f - g^{\delta}\|_{L^2(J)} \leq \delta, | f|_{TV} \leq \kappa \}.$$
Then, the diameter of $S(\delta,g^{\delta})$ tends to zero as $\delta \to 0$ (at a rate of the order $|\log \delta|^{-1/2}$).
\end{cor}

Thus, under the assumption that the true solution $f_{\text{ex}}$ to $H_T f = g$ has bounded variation, any algorithm that, given $\delta$ and $g^{\delta}$, finds a solution in $S(\delta, g^{\delta})$, is a regularization method. 

As with Theorem \ref{thm2}, we are again able to improve on Theorem \ref{thm3} by assuming $f$ is sufficiently smooth, in which case the inequality approaches in the limit an inequality that is in certain cases as strong as the conjecture \eqref{thm2_ineq_conj}.
\begin{thm}\label{thm3a} Let $I, J \subset \mathbb{R}$ be intervals in the configuration of Case 3 or Case 4.   
Then there exists an order 2 differential operator $L_I$ (defined in \S \ref{sec:diff-operator}) such that for any $M \geq 1$ and any $f \in C_c^{2M+1}(I)$, 
$$ \|H f\|_{L^2(J)} \geq c_{1,M} \exp{\left(-c_{2,M} \left( \frac{ |L_I^M f|_{TV}}{\|f\|_{L^2(I)}}\right)^{\frac{2}{4M+1}} \right)} \| f \|_{L^2(I)} ,$$
for some constants $c_{1,M}, c_{2,M} > 0$ depending only on $I, J$ and $M,$ with the property that as $M \maps \infty$, $c_{1,M},c_{2,M}$ tend to finite limits $c_1,c_2>0$. 
\end{thm}
Note that Theorem \ref{thm3a} reduces to Theorem \ref{thm3} for $M=0$.
It is again instructive to consider an example. For this purpose we can take $f_N(x) = \sin (2\pi Nx)$ with the interval $I = (0,1)$ as before, and again use dilation and convolution with a compactly supported $C^{\infty}$ bump function to bring $f_N$ into $C_c^{2M+1}(I)$.
Then, morally speaking, the theorem implies
\begin{align*}
  \|H f_N\|_{L^2(J)} &\geq c_{1,M} \exp{\left(-c_{2,M} \left( \frac{(2\pi N)^{2M+1}\| \cos(2\pi Nx) \|_{L^1(I)}}{\| f_N \|_{L^2(I)}}\right)^{\frac{2}{4M+1}} \right)} \|  f_N\|_{L^2(I)} \\
  & \geq c_1 \exp{\left(-c_2' N\right)} \|  f_N\|_{L^2(I)}
  \end{align*}
as $M \maps \infty$; the relation (\ref{relation}) shows this is as strong as (\ref{thm2_ineq_conj}). \\

The proofs of both Theorem \ref{thm2} and Theorem \ref{thm3} (see Sections \ref{sec:Proof-Thm2} and \ref{sec:Proof-Thm3}) are in a similar spirit and hinge on $TT^*$ arguments in combination with an 
eigenfunction decomposition of $TT^*$. The eigenfunctions are well understood; the difficulty is in putting
this information to use in the most effective way.
The proof of Theorem \ref{thm2} uses their orthogonality and the fact that an associated differential operator is comparable to $-\Delta$, but 
does not rely on the asymptotic behavior of the eigenfunctions (merely on asymptotics of the eigenvalues).
In contrast, the proof of Theorem \ref{thm3} uses an elementary estimate adapted to the eigenfunctions and inspired from
classical Fourier analysis: this estimate is sharp but not sophisticated enough to capture complicated behavior at
different scales simultaneously. It is not clear to us whether and how these arguments could be refined.

\subsection{An improved estimate for Case 4.}\label{sec:results-overlap}

Case 3, with disjoint intervals $I$ and $J$, is the worst case scenario in terms of reconstruction from Hilbert transform data. It seems that reconstruction in Case 4, the truncated Hilbert transform with overlap, is an easier task in the sense that one would expect the inversion problem to be more stable. The singular values decay to zero at a similar exponential rate in both cases, since the Hilbert transform with overlap contains, at this level of generality, the Hilbert transform with a gap as a special case (acting on functions supported away
from $I \intersect J$). It is this ill-posedness that in practice has led to the concept of \textit{region of interest} reconstruction. Here, the aim is to reconstruct the function $f$ only on the region where the Hilbert transform has been measured. For the truncated Hilbert transform with overlap this means reconstruction of $f$ only on the overlap region $I \cap J$. 

The reason this problem of partial reconstruction inside $I \intersect J$ may be more stable has an intuitive explanation: one would expect interaction with the singularity of the Hilbert transform
to be such that it cannot lead to significant cancellation. {More formally, one can consider the singular value decomposition of $H_T$. In the case where $I \intersect J \neq \emptyset$, the singular values accumulate at both $0$ and $1$. Moreover, the singular functions have the property that they oscillate on $I \cap J$ and are monotonically decaying to zero on $I \backslash J$ as the singular values accumulate at $1$. The opposite is true when the singular values decay to zero: the corresponding singular functions oscillate on $I \backslash J$, i.e., outside of the region of interest, and are monotonically decaying to zero on $I \cap J$. (For a proof of these properties we refer to \cite{erdelyi}.) Figure \ref{fig:sing-func} below illustrates the behavior of the singular functions for a specific choice of overlapping intervals $I$ and $J$.

\begin{figure}[ht!]
  \begin{center}

        \subfigure{
 \label{fig:sing-func-0}
           \includegraphics[width=0.45\textwidth]{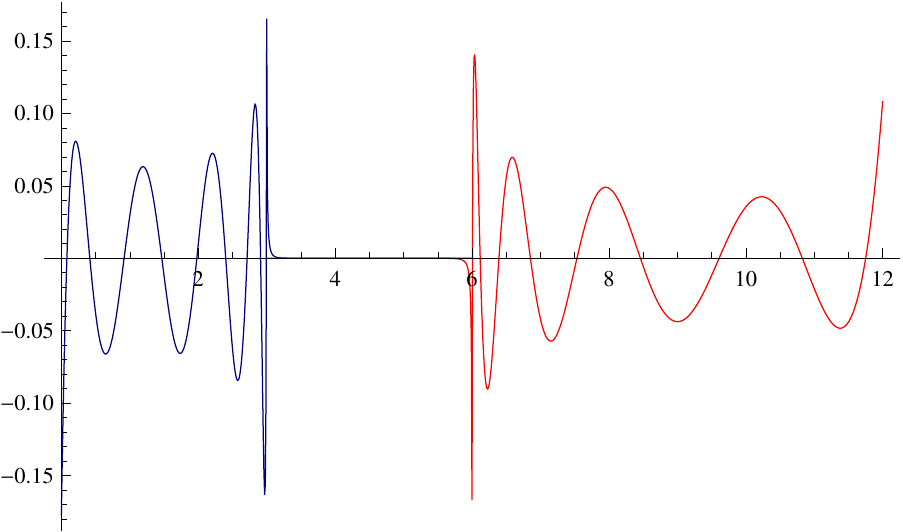}
     }
      \subfigure{
           \label{fig:sing-func-1}
            \includegraphics[width=0.45\textwidth]{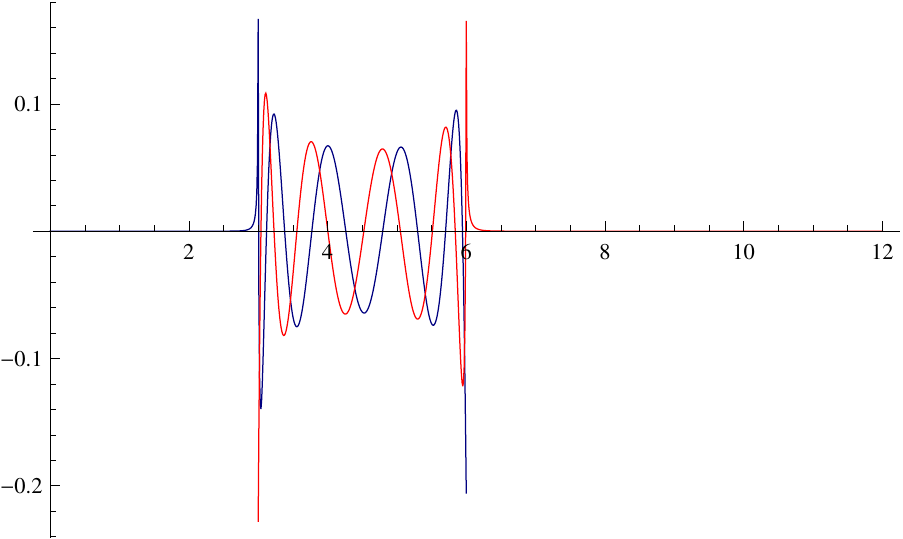}
        }
   \end{center}
    \caption{Examples of singular functions $u_n$ (red) and $v_n$ (blue) for the overlap case $I=(0,6)$, $J=(3,12)$. \textit{Left:} For $\sigma_n$ close to 0, the singular functions are exponentially small on $(3,6)$ and oscillate outside of $(3,6)$. \textit{Right:} For $\sigma_n$ close to 1, the functions oscillate on $(3,6)$ and are exponentially small outside of the overlap region.
   }
   \label{fig:sing-func}
\end{figure}

A more precise estimate on the decaying part of the singular functions is the subject of joint work by the first author with M. Defrise and A. Katsevich \cite{adk2}. Let $I = (a_2,a_4)$ and $J=(a_1,a_3)$ for real numbers $a_1<a_2<a_3<a_4$,  and let us consider the singular functions $u_n$ on $I$ corresponding to the singular values $\sigma_n$ decaying to zero. Then, one can show that for any $\mu>0$ there exist positive constants $B_{\mu}$ and $\beta_{\mu}$ such that 
\[\|u_n\|_{L^2([a_2,a_3-\mu])} \leq B_{\mu} e^{-\beta_{\mu} n}\]
 for sufficiently large index $n$.
Exploiting this property, we can eliminate the dependence  in Theorem \ref{thm3} on the variation of $f$ within the region of interest.

\begin{thm}\label{thm4}
Let $J=(a_1,a_3)$ and $I=(a_2,a_4) \subset \mathbb{R}$ be open intervals  with $a_1< a_2<a_3<a_4$. Fix a closed subinterval $J^* = [a_1^*,a_3^*] \subset J$ with $a_1^* < a_2 < a_3^*$. Then for any function $f \in W^{1,1}(I)$
such that there exists at least one point $x_0 \in I\setminus J^*$ at which $f(x_0)=0$,  the following holds:
$$ \|H f\|_{L^2(J)} \geq c_1 \exp{\left(-c_2\frac{ |\chi_{I \backslash J^*} f|^2_{TV}}{\|f\|^2_{L^2(I)}}\right)} \| f \|_{L^2(I)} ,$$
for some constants $c_1, c_2 > 0$ depending only on $I, J$ and $J^*$. 
\end{thm}

\textit{Remark.}
Theorem \ref{thm4} can be used in a similar fashion as Theorem \ref{thm3} to obtain a stability estimate analogous to Corollary \ref{cor-stab}. As prior knowledge we assume $|\chi_{I \backslash J^*} f_{ex}|_{TV} \leq \kappa$ and $\int_I f_{ex} = C$. Then, the statement can be formulated similarly as before, with the only change that the set of admissible solutions becomes
$$S(\delta, g^{\delta}) = \{ f \in W^{1,1}(I): \| H_T f - g^{\delta}\|_{L^2(J)} \leq \delta, |\chi_{I \backslash J^*} f|_{TV} \leq \kappa, \int_I f = C \}.$$
One can then adapt the proof of Corollary \ref{cor-stab} to obtain that the diameter of $S(\delta, g^{\delta})$ tends to zero as $\delta \to 0$ at a similar rate as before of the order $|\log \delta|^{-1/2}$. The only difference to Corollary \ref{cor-stab} is that now the constants in (\ref{cor-bound}) depend not only on $I$ and $J$, but also on $J^*$. 

\textit{Remark.}
Under the assumption that $f$ does not vanish on $I$, we can improve on Theorem \ref{thm4}, giving a lower bound with polynomial decay; see remarks following Lemma \ref{polydecay}.
A stronger version of Theorem \ref{thm4} for smoother functions can also be derived by an iterated argument, analogous to the adaptation of Theorems \ref{thm2a} and \ref{thm3a} from the proofs of Theorem \ref{thm2} and \ref{thm3}; we omit the details.

An interesting  question that remains open is whether estimates of the form 
\beq\label{Hfh}
 \|H f\|_{L^2(J)} \geq h\left(\frac{\left| f \right|_{\text{TV}}}{\|f\|_{L^2(I)}}\right)\|f\|_{L^2(I \cap J)}
 \eeq
are possible for a function $h$ that shows a decay that is slower than the quadratically exponential type in Theorems \ref{thm3} and \ref{thm4}, yet does not introduce a differential operator such as $L_I$. Note that (\ref{Hfh}) would give a lower bound on $\| H f\|_{L^2(J)}$ with respect to $\|f\|_{L^2(I \intersect J)}$ instead of $\| f \|_{L^2(I)}$, which is why we could expect such a function $h$ to decay slower than in Theorems  \ref{thm3} and \ref{thm4}: if $f$ is mainly supported on $I \backslash J$, i.e., away from the overlap, we will most likely not be able to improve on the conjecture \eqref{thm2_ineq_conj}. If, however, $f$ has a significant portion of its $L^2$--mass inside the overlap $I \intersect J,$ then $\|H f\|_{L^2(J)}$ cannot be too small. In terms of a possible stability estimate this implies that a regularization method guarantees good recovery only within the overlap $I \intersect J.$ Such a stability estimate would be of particular interest, since in practice one only aims at reconstruction within the overlap (i.e. the region of interest).\\

\subsection{A word on the proofs}
 We note in advance that the results of Theorems \ref{thm2} to \ref{thm3a} are such that the statements for the truncated Hilbert transform with overlap  follow from the corresponding statement for the truncated Hilbert transform with a gap. Indeed, in Case 4, since  $J \not\subset I$, we can always find an interval 
$J^* \subset J$ such that $I$ and $J^{*}$ are disjoint. Trivially, however,
\beq\label{JJ*}
 \|Hf\|_{L^2(J)} \geq \|Hf\|_{L^2(J^*)},
 \eeq
so that a lower bound for $\|Hf\|_{L^2(J^*)}$ suffices.
Therefore, in our proofs of Theorems \ref{thm2} to \ref{thm3a}, we may restrict ourselves to Case 3, i.e., the truncated Hilbert transform with a gap.

\section{Proof of Theorem \ref{thm1}} 

\begin{proof}
Consider all $g \in BV(I)$ and define for each such $g$ the corresponding function $\tilde{g} = g/\| g\|_{L^2(I)}$, so that $| \tilde{g}|_{TV} = |g|_{TV}/\|g \|_{L^2(I)}$. We will show that for any fixed $\kappa>0$ if we consider all such normalized $\tilde{g}$ for which $|\tilde{g}|_{TV} \leq \kappa$, then there exists $c_0(\kappa)>0$ such that 
\beq\label{H_nonzero}
 \| H \tilde{g} \|_{L^2(J)} \geq c_0(\kappa).
 \eeq
From this we may conclude that there exists a positive-valued function $h$ such that 
\[  \| H g \|_{L^2(J)} \geq h \left( \frac{ | g |_{TV}}{\|g \|_{L^2 (I)}} \right) \|g \|_{L^2(I)}.\]

The proof of (\ref{H_nonzero}) will proceed by contradiction. We begin by assuming the existence of a sequence $f_n \in BV(I)$ that has uniformly bounded variation $|f_n|_{TV} \leq \kappa$, uniform norm $\| f_n\|_{L^2(I)} =1$, and such that $\| H f_n\|_{L^2(J)}$ is not bounded below, i.e., 
\begin{equation}\label{Hf-to-zero}
 \lim_{n \rightarrow \infty} {\|Hf_n\|_{L^2(J)}} = 0.
\end{equation}
\textbf{Step 1.} The first step of the proof consists of showing that these assumptions imply the uniform boundedness of $f_n$, more precisely that the following holds:
\begin{equation}\label{limsup-fn}
\limsup_{n \rightarrow \infty}{\|f_n\|_{L^{\infty}(I)}} \leq \kappa +  |I|^{-\frac{1}{2}}.
\end{equation}
Suppose that for some index $N$ and some $\varepsilon \in (0,|I|^{-\frac{1}{2}})$, we have $\|f_n \|_{L^{\infty}(I)} \leq \kappa+\varepsilon$ for all $n \geq N$. 
Then, we have found a sequence that is uniformly bounded with the above bound in \eqref{limsup-fn}. If such an index $N$ does not exist, we can find a subsequence $f_{n_k}$ such that
$$\| f_{n_k}\|_{L^{\infty}(I)} > \kappa + \varepsilon.$$
This together with the assumed bound on $|f_n|_{TV}$, requires that $f_{n_k}$ does not change sign. Suppose w.l.o.g. that $f_{n_k} > 0$. Then, $\left| f_{n_k} \right|_{TV} \leq \kappa$ implies
$$0 < \| f_{n_k} \|_{L^{\infty}(I)} - \kappa \leq f_{n_k}(x), \quad x \in I.$$
Hence,
\[\int_I (\| f_{n_k} \|_{L^{\infty}(I)} - \kappa )^2 dx \leq \int_I f_{n_k}(x)^2 dx ,\]
which shows that
\[(\| f_{n_k} \|_{L^{\infty}(I)} - \kappa )^2 \cdot |I|  \leq 1,\]
and therefore
\[ \| f_{n_k} \|_{L^{\infty}(I)} \leq  \kappa + |I|^{-\frac{1}{2}}.\]

\textbf{Step 2.} This step relies on Helly's selection theorem\index{Helly's selection theorem}, which is a compactness theorem for $BV_{loc}.$
Let $\Omega \subset \mathbb{R}$ be an open set and $f_n:\Omega \rightarrow \mathbb{R}$ a sequence of
functions with
$$\sup_{n \in \mathbb{N}}{\left(\|f_n\|_{L^1(\Omega)} + \left\| \frac{d}{dx}f_n\right\|_{L^1(\Omega)}\right)} < \infty,$$
where the derivative is taken in the sense of tempered distributions. 
Then there exists a subsequence $\{f_{n_k}\}$ and a function $f \in BV_{loc}(\Omega)$ such that $f_{n_k}$ converges to $f$ pointwise and in $L_{loc}^1(\Omega)$. Moreover, $|f|_{TV} \leq \liminf_{n \to \infty} |f_n|_{TV}.$
Applying Helly's selection theorem to our sequence $\{f_n\}$ implies the existence of a subsequence $\{ f_{n_k} \},$ such that their pointwise limit $f$ is in $BV(I).$
 Furthermore, the uniform boundedness established in Step 1 yields that for each $n_k$, 
$$|f_{n_k}(q)| \leq\| f_{n_k}\|_{L^{\infty}(I)} \leq \kappa + |I|^{-\frac{1}{2}}.$$
 Moreover, the dominated convergence theorem implies that the uniform boundedness of $f_{n_k}$ and $f$, together with their pointwise convergence to $f$ results in convergence in the $L^2$-sense, i.e. 
\beq\label{limit0}
\| f_{n_k} - f \|_{L^2(I)} \to 0.
\eeq
We recall the simple observation that the truncated Hilbert transform remains bounded on $L^2$, since 
\[  \| Hf\|_{L^2(J)} = \| \mathcal{P}_J H f\|_{L^2(\R)} \leq \|  H f\|_{L^2(\R)} = \| f\|_{L^2(I)}.\]
Consequently, from (\ref{limit0}) we deduce
$$\| H f_{n_k} - H f \|_{L^2(J)} \leq \| f_{n_k} - f \|_{L^2(I)} \to 0.$$
Combining this with \eqref{Hf-to-zero} yields $\| H f\|_{L^2(J)}=0.$
Lemma 5.1. in \cite{reema1} states that if $f \in L^2(I)$ and $Hf$ vanishes on an open subset away from $I$, then $f \equiv 0.$ This contradicts the assumption $\| f_n \|_{L^2(I)}=1$ and hence completes the proof.
\end{proof}

\section{A differential operator}\label{sec:diff-operator}
Our proofs of the remaining theorems make essential use of the singular value decomposition of $H_T$. Using an old idea of Landau, Pollak and Slepian \cite{landau1961, landau1962, slepian1961} (and later of Maass in the context of tomography \cite{ma})
in the form of Katsevich \cite{kat, kat2}, we use an explicit differential operator to establish
a connection to the singular value expansion. The explicit form of the involved operators will allow
us to deduce that if $\|f_x\|_{L^2(I)}$ is small, then there is some explicit part of the $L^2$--norm of
$f$ that is comprised of singular functions associated to the largest singular values.\\

Let $H_T$ denote the truncated Hilbert transform with a gap, so that we may assume that $J = (a_1, a_2)$ and $I = (a_3, a_4)$ for real numbers $a_1<a_2<a_3<a_4$. 
Let $\{\sigma_n ; u_n, v_n\}$ be the singular value decomposition of $H_T$. Note that, by definition, for $\|f\|_{L^2(I)} = 1$,
$$\| H_T f \|_{L^2(J)}^2 = \sum_{n=0}^{\infty}{ |\langle f,u_n \rangle |^2 \sigma_n^2 }.$$
Following Katsevich \cite{kat}, we define the differential form 
$$ (L \psi)(x):= (P(x)\psi_x(x))_x+2 (x-\sigma)^2 \psi(x),$$
where
$$
P(x) = \prod_{i=1}^4 (x-a_i) \qquad \mbox{and} \qquad  \sigma = \frac{1}{4}\sum_{i=1}^4 a_i.
$$
For a correct definition of an unbounded operator it is necessary to indicate the domain it is acting on, as unbounded operators cannot be defined on all of $L^2$ (Hellinger--Toeplitz theorem). Therefore, we let $AC_{loc}(I)$ denote the space of locally absolutely continuous functions on $I$ and define the domains 
\begin{equation*}
D_{\max} := \{ \psi: I \rightarrow \mathbb{C}: \psi, P \psi_x \in AC_{loc}(I); \psi, L \psi \in L^2(I) \}
\end{equation*} 
and 
$$\mathcal{D} = \{ \psi \in D_{\max}:  P(x) \psi_x(x) \to 0 \text{ for } x \to a_3^+, x \to a_4^- \}.$$ 
We let $L_I$ be the restriction of $L$ to the domain $\mathcal{D}$ and note that $L_I$ is a self-adjoint operator \cite{zettl}.

Then, as shown in \cite{kat}, a commutation property of $L_{I}$ with $H_T$ proves that the functions $\{u_n\}$ form an orthonormal basis of $L^2(I)$ and that they are the eigenfunctions of $L_{I}$, that is $L_{I} u_n = \lambda_n u_n$ with $\lambda_n$ being the $n$-th eigenvalue of $L_{I}$. In addition, the asymptotic behavior as $n \to \infty$ of the eigenvalues $\lambda_n$ of $L_{I}$  as well as that of the singular values $\sigma_n$ of $H_T$ is known. Katsevich \& Tovbis \cite{kat3}
have given the asymptotics as $n \to \infty$ including error terms, from which we can deduce that for all $n \in \mathbb{N}$
\begin{align}
 \lambda_n &\geq k_1 n^2 \label{lamnbd} \\
 \sigma_n & \geq e^{-k_2 n}, \label{signbd}
\end{align}
where $k_1, k_2 > 0$ depend only on the intervals $I$ and $J$. 

We must also consider the $M$-th iterate $L_I^M$ of $L_I.$ For $M \in \mathbb{N},$ let $D(L_I^M)$ denote the domain of the self-adjoint operator $L_I^M.$ Then, we define the following sets of functions in $L^2(I):$
\begin{equation*}
A_M = \{ f \in L^2(I): f \in D(L_I^{M+1}) \text{ and } L_I^M f \in H^1(I)\}.
\end{equation*}
We note that these classes of functions are dense in $L^2(I)$ and that $C_c^{2M+1}(I)$ is a subset of $A_M$. Also, $A=\bigcap_{M=1}^{\infty} A_M$ is dense in $L^2(I)$ and $C_c^{\infty}(I) \subset A.$\\

\textit{Remark.} The asymptotics in the results of Katsevich \& Tovbis \cite{kat3} are actually more precise than stated. In particular, setting
$ I = (a_3, a_4)$ and $J = (a_1,a_2)$ and
\begin{align*}
 K_{+} &= \frac{\pi}{\sqrt{(a_4-a_2)(a_3-a_1)}} {}_{2}F_{1}\left(\frac{1}{2}, \frac{1}{2}, 1;
\frac{(a_3-a_2)(a_4-a_1)}{(a_4-a_2)(a_3-a_1)}\right) \\
 K_{-} &= \frac{\pi}{\sqrt{(a_4-a_2)(a_3-a_1)}} {}_{2}F_{1}\left(\frac{1}{2}, \frac{1}{2}, 1;
\frac{(a_2-a_1)(a_4-a_3)}{(a_3-a_1)(a_4-a_2)}\right),
\end{align*}
where ${}_{2}F_{1}$ is the hypergeometric function, Katsevich \& Tovbis derive
\begin{align}
 \lambda_n &= \frac{\pi^2}{K_{+}^2} n^2 (1+o(1)) \nonumber \\
 \sigma_n &= e^{- \pi \frac{K_+}{K_-} n}(1+o(1)) \nonumber
\end{align}
for sufficiently large $n$. These asymptotic relations allow one to state (\ref{lamnbd}) and (\ref{signbd}) for all $n \geq N_0$, for some $N_0 \in \mathbb{N}$ depending on $I$ and $J$. One can then find explicit constants $k_1$ and $k_2$ depending on $I$ and $J$ such that relations (\ref{lamnbd}) and (\ref{signbd}) hold for all $n \in \mathbb{N}$. Similarly, exploiting the asymptotics one can derive an upper bound of the form 
\begin{equation}\label{signupbd}
\sigma_n \leq \tilde{K}_2 e^{-K_2 n}
\end{equation}
with $K_2, \tilde{K}_2$ depending only on $I$ and $J$.

\section{Proof of Theorems \ref{thm2} and \ref{thm2a}}\label{sec:Proof-Thm2}

We now turn to the proof of Theorem \ref{thm2}, for which we exploit the following density argument. Since $H^2(I)$ is dense in $H^1(I)$ w.r.t. the $H^1$ topology and, as can be easily verified, $H^2(I) \subset \mathcal{D},$ one can conclude that $H^1(I) \cap \mathcal{D}$ is dense in $H^1(I)$ w.r.t. the $H^1$ topology. Thus, it suffices to prove the statement of Theorem \ref{thm2} for functions $g$ in $H^1(I) \cap \mathcal{D}$; for each such function we normalize it to $\tilde{g} = g/\|g \|_{L^2(I)}$, so that to prove the theorem it would suffice to show that 
\[ \|H \tilde{g} \|_{L^2(J)} \geq c_1 \exp(-c_2 \| \tilde{g}_x \|_{L^2(I)}).\]

We now therefore assume we have $f \in H^1(I) \cap \mathcal{D}$ with $\| f \|_{L^2(I)} =1$.
Integration by parts yields that for $f \in H^1(I) \cap \mathcal{D}$,

\begin{align*}
\langle L_I f, f \rangle &= - \int_{a_3}^{a_4} P(x) f_x(x)^2 dx   + (P(x) f_x(x)) f(x) \Big|_{a_3}^{a_4} + \int_{a_3}^{a_4} 2 (x-\sigma)^2 f(x)^2 dx\\
&= - \int_{a_3}^{a_4} P(x) f_x(x)^2 dx  + \int_{a_3}^{a_4} 2 (x-\sigma)^2 f(x)^2 dx 
\end{align*}          
so that 
      \begin{align}
|\langle L_I f, f \rangle |
&\leq \|P\|_{L^{\infty}(I)} \|f_x\|_{L^2(I)}^2 + 2(a_4-a_1)^2 \|f\|_{L^2(I)}\nonumber \\
&\leq  k_3\|f_x\|_{L^2(I)}^2 + k_3, \nonumber
\end{align}                          
for some constant $k_3 > 0$ depending only on $I$ and $J$. Altogether, we thus have
\begin{equation}\label{L_upper}
  \sum_{n=0}^{\infty} |\langle f,u_n \rangle |^2 \lambda_n = | \langle L_I f, f \rangle | \leq k_3\|f_x\|_{L^2(I)}^2 + k_3.
\end{equation}
Hence for any $N \geq 1$, it follows from the asymptotic behavior $\lam_n \geq k_1n^2$ that
\begin{eqnarray}
 1 = \| f\|_{L^2(I)}^2 &=&  \sum_{n=0}^{N} |\langle f,u_n \rangle |^2 +  \sum_{n=N+1}^{\infty} |\langle f,u_n \rangle |^2 \nonumber \\
 &  \leq &  \sum_{n=0}^{N} |\langle f,u_n \rangle |^2 +  \sum_{n=N+1}^{\infty} |\langle f,u_n \rangle |^2  \frac{\lam_n}{k_1n^2} \nonumber \\
  &  \leq &  \sum_{n=0}^{N} |\langle f,u_n \rangle |^2 + k_1^{-1}N^{-2} \sum_{n=N+1}^{\infty} |\langle f,u_n \rangle |^2  \lam_n ,\label{iterate_1} 
 \end{eqnarray}
 so that by (\ref{L_upper}),
 \[ \sum_{n=0}^N |\langle f,u_n \rangle |^2 \geq 1- k_3 k_1^{-1}N^{-2} \left( \| f_x \|_{L^2(I)}^2 +1\right).\]
 Hence choosing the least integer $N$ such that 
 \[ N^2  \geq 2 k_3 k_1^{-1} (\|f_x\|_{L^2(I)} + 1)^2 \geq 2  k_3k_1^{-1}(\|f_x\|_{L^2(I)}^2 + 1),\]
and setting $k_4=(2k_3 k_1^{-1})^{1/2}$ yields
$$  \sum_{n \leq \lceil k_4 (\|f_x\|_{L^2(I)}+1)\rceil }  |\langle f,u_n \rangle |^2 \geq \frac{1}{2}.$$
Then, however,
\begin{align*}
 \| H_T f \|_{L^2(J)}^2 &=  \sum_{n=0}^{\infty}{ |\langle f, u_n \rangle |^2 \sigma_n^2 }   \\
&\geq  \sum_{n \leq \lceil k_4(\|f_x\|_{L^2(I)}+1) \rceil}{ |\langle f, u_n \rangle |^2 \sigma_n^2 }        \\
&\geq  \left( \sum_{n \leq \lceil k_4(\|f_x\|_{L^2(I)}+1) \rceil}{ |\langle f, u_n \rangle |^2  } \right) \sigma_{ \lceil k_4 (\|f_x\|_{L^2(I)}+1) \rceil}^2 \\
&\geq  \frac{1}{2} e^{-2 k_2 k_4 \|f_x\|_{L^2(I)}-2k_2(k_4+1)} \geq     k_5 e^{-2 k_2 k_4\|f_x\|_{L^2(I)}}                                                      
\end{align*}

for some constant $k_5>0$, as desired.

We now modify the above argument to prove the stronger result of Theorem \ref{thm2a} when $f \in A_M$ for some arbitrary $M \in \mathbb{N}$.
We start with the observation that for $\|f\|_{L^2(I)}=1$ and any $N \geq 1$,
\begin{eqnarray}
 1 &=&  \sum_{n=0}^{N} |\langle f,u_n \rangle |^2 +  \sum_{n=N+1}^{\infty} |\langle f,u_n \rangle |^2 \nonumber \\
 &  \leq &  \sum_{n=0}^{N} |\langle f,u_n \rangle |^2 +  \sum_{n=N+1}^{\infty} |\langle f,u_n \rangle |^2 \left( \frac{\lam_n}{k_1n^2} \right)^{2M+1} \nonumber \\
  &  \leq &  \sum_{n=0}^{N} |\langle f,u_n \rangle |^2 + (k_1 N^2)^{-(2M+1)} \sum_{n=N+1}^{\infty} |\langle f,\lam_n^M u_n \rangle |^2 \lam_n \nonumber \\
    & = &  \sum_{n=0}^{N} |\langle f,u_n \rangle |^2 +(k_1 N^2)^{-(2M+1)}\sum_{n=N+1}^{\infty} |\langle f,L_I^M u_n \rangle |^2  \lam_n\nonumber \\
  &  = &  \sum_{n=0}^{N} |\langle f,u_n \rangle |^2 + (k_1 N^2)^{-(2M+1)}\sum_{n=N+1}^{\infty} |\langle L_I^M f,u_n \rangle |^2  \lam_n .\label{iterate_1a'} 
 \end{eqnarray}

 Now we recall that there is a constant $k_3>0$ such that for $g \in \mathcal{D} \intersect H^1(I)$, 
 \[  \sum_{n=0}^\infty |\langle g,u_n \rangle|^2 \lam_n = |\langle L_I g,g \rangle| \leq k_3 \|g_x \|_{L^2(I)}^2 + k_3.\]
 We apply this with $g= L_I^M f$, to conclude from (\ref{iterate_1a'}) that 
 \[   \sum_{n=0}^{N} |\langle f,u_n \rangle |^2 \geq 1 - 
 (k_1 N^2)^{-(2M+1)} k_3 (\| (L_I^M f)_x \|^2_{L^2(I)}+1).\]
 Hence choosing the least integer $N$ such that 
 \begin{align*} N  &\geq \big(2 k_3k_1^{-(2M+1)}(\|(L_I^M f)_x \|_{L^2(I)} + 1)^2\big)^{\frac{1}{2(2M+1)}} \\
& \geq \big(2 k_3k_1^{-(2M+1)}(\|(L_I^M f)_x \|_{L^2(I)}^2 + 1)\big)^{\frac{1}{2(2M+1)}},
\end{align*}
 we see that for $k_4 =(2k_3)^{1/2(2M+1)}k_1^{-1/2},$
$$  \sum_{n \leq \lceil k_4 (\| (L_I^M f)_x \|_{L^2(I)}^{1/(2M+1)}+1) \rceil} |\langle f,u_n \rangle |^2 \geq \frac{1}{2}.$$
As before, we now obtain  a lower bound
\begin{align*}
 \| H_T f \|_{L^2(J)}^2 &=  \sum_{n=0}^{\infty}{ |\langle f, u_n \rangle |^2 \sigma_n^2 }   \\
&\geq  \sum_{n \leq \lceil k_4 (\| (L_I^M f)_x \|_{L^2(I)}^{1/(2M+1)}+1) \rceil}{ |\langle f, u_n \rangle |^2 \sigma_n^2 }        \\
&\geq  \left( \sum_{n \leq \lceil k_4 (\| (L_I^M f)_x \|_{L^2(I)}^{1/(2M+1)}+1) \rceil}{ |\langle f, u_n \rangle |^2  } \right) \sigma_{\lceil k_4(\| (L_I^M f)_x \|_{L^2(I)}^{1/(2M+1)}+1)\rceil}^2 \\
&\geq  \frac{1}{2} \exp(-2 k_2 k_4 \| (L_I^M f)_x \|_{L^2(I)}^{1/(2M+1)}-2k_2(k_4+1)) \\
&  \geq     k_5 \exp(-2 k_2 k_4\|(L_I^M f)_x\|_{L^2(I)}^{1/(2M+1)})                                                    
\end{align*}
for some constant $k_5=(1/2)e^{-2k_2(k_4+1)}$. We need only note that as $M \maps \infty$, $k_4 \maps k_1^{-1/2}$, and $k_5 \maps (1/2)e^{-2k_2(k_1^{-1/2}+1)}$, both positive finite limits.

\section{Proof of Theorems \ref{thm3} and \ref{thm3a}.} \label{sec:Proof-Thm3}
It is well known that smoothness of a function $f:\mathbb{T}\rightarrow \mathbb{R}$ translates into decay of the Fourier coefficients $\hat{f}(n)$. This statement is usually proven using integration by parts; in particular, $f \in C^k$ yields $|\hat{f}(n)| \leq C_{f^{(k)},k}n^{-k}$. However, it is easy to see that
for $k=1$ it actually suffices to require $f$ to be of bounded variation: this observation dates back at least to a paper from 1967 (but is possibly quite a bit older) of Taibleson \cite{tai},
who showed that
$$ |\hat{f}(n)| \leq 2\pi \frac{\left| f \right|_{\text{TV}}}{n}.$$
We will show the analogous statement with the Fourier system replaced by the singular functions $u_n$ of the operator $L_I$; the argument exploits an asymptotic expression and,
implicitly, Abel's summation formula as a substitute for integration by parts.

\begin{lemma}\label{lemma_un}
Let $I$ and $J$ be disjoint finite open intervals on $\R$. There exists $c > 0$ depending only on the intervals $I, J$ such that for any $f$ of bounded variation that is supported on $I$ and vanishes at the boundary of the interval $I$,
$$ \left| \left\langle f, u_n\right\rangle \right| \leq c\frac{\left| f \right|_{\text{TV}}}{n}.$$
\end{lemma}
\begin{proof}
 We may assume w.l.o.g. by density that $f \in C^1$ (or, alternatively, replace every
integral by summation, and integration by parts by Abel's summation formula). Let $I=(a_3,a_4).$ It suffices to show that
\begin{equation}\label{int-un}
 \forall~x\in (a_3, a_4): \qquad  \left| \int_{a_3}^{x}{u_n(z)dz}\right| \leq \frac{c}{n}.
\end{equation}
Once this is established (see the appendix in \S \ref{sec_int-un} for the proof of the above statement), we can write
\begin{align*}
\left| \int_{a_3}^{a_4}{f(x) u_n(x) dx} \right|&= \left| \int_{a_3}^{a_4}{f(x) \left(\int_{a_3}^{x}{u_n(z)dz}\right)_x dx} \right| \\
&= \left| \int_{a_3}^{a_4}{f_x(x) \left(\int_{a_3}^{x}{u_n(z)dz}\right) dx} \right| \\
&\leq \sup_{a_3 \leq x \leq a_4}{\left| \int_{a_3}^{x}{u_n(z)dz}\right|}\int_{a_3}^{a_4}{|f_x(x)|dx},
\end{align*}
in which the boundary terms vanish by the assumption on $f$.
\end{proof}

\subsection{Proof of Theorem \ref{thm3}}
This section is split into two parts: we first assume that there exists a point $x_0 \in I$ such that $f(x_0) = 0$, and argue using that property. The second part of the section is completely independent
and establishes a stronger result in the case that $f$ does not change sign.

In the first case, given $g \in W^{1,1}(I)$ we consider the normalization $\tilde{g} = g/ \|g \|_{L^2(I)}$, so that it would suffice to show that under the hypotheses of Theorem \ref{thm3},
\beq\label{Hgtilde}
\| H \tilde{g} \|_{L^2(I)} \geq c_1 \exp (-c_2 |\tilde{g}|^2_{TV} ).
\eeq
Thus we now consider $f \in W^{1,1}(I)$  with $\|f\|_{L^2(I)} =1$ and such that $f$ vanishes at least at one point in $I$. If $f$ vanishes at the endpoints of $I$, we may apply Lemma \ref{lemma_un} directly; otherwise we use Lemma \ref{lemma_approx} to approximate $f \in BV(I)$ by a sequence of $f_n \in C_c^\infty(I)$ (in particular, vanishing at the boundary of $I$) such that $\|f_n - f\|_{L^2(I)} \maps 0$ and $|f_n|_{TV} \leq 3 |f|_{TV}$. Then if we prove (\ref{Hgtilde}) for each $f_n$ we can conclude it holds for $f$, since 
\beq\label{f_approx}
 c_1 \exp (-9c_2 |f|^2_{TV} ) \leq  \| H f_n\|_{L^2(J)} \leq  \| H f\|_{L^2(J)}  +   \| H (f-f_n)\|_{L^2(J)} ,
 \eeq
and 
$$\| H (f-f_n)\|_{L^2(J)} \leq C \|f - f_n\|_{L^2(I)} \maps 0 \quad \mbox{as} \quad n \maps \infty.$$
We may now assume that $f$ vanishes at the boundary of $I$, and note that by Lemma \ref{lemma_un},
\begin{align*}
 1 = \|f\|_{L^2(I)}^2
&= \sum_{n=0}^{N}{\left| \left\langle f, u_n \right\rangle\right|^2}  +  \sum_{n=N+1}^{\infty}{\left| \left\langle f, u_n \right\rangle\right|^2}\\
&\leq \sum_{n=0}^{N}{\left| \left\langle f, u_n \right\rangle\right|^2} + c^2 \left| f \right|_{\text{TV}}^2\sum_{n=N+1}^{\infty}{\frac{1}{n^2}} \\
&\leq  \sum_{n=0}^{N}{\left| \left\langle f, u_n \right\rangle\right|^2} + \frac{c^2 \left| f \right|_{\text{TV}}^2}{N}.
\end{align*}
This implies that at least half of the $L^2$--mass is contained within the first $N = \lceil 2 c^2 \left| f \right|_{\text{TV}}^2 \rceil$ frequencies. The
remainder of the argument can be carried out as in Theorem \ref{thm2}.

It remains to show that we can actually restrict ourselves to the case where $f(x_0) = 0$ for some $x_0 \in I$. Assume now that we are in Case 3 (the argument
for Case 4 follows completely analogously by reducing it to Case 3 using (\ref{JJ*})). It is not difficult to see that we get a much stronger inverse inequality (with a polynomial instead
of a superexponential decay).
\begin{lemma}\label{polydecay} Let $I, J$ be as in Case 3 and assume that $f$ has no root on $I$. Then,
$$ \| Hf \|_{L^2(J)} \geq  \frac{ |J|^{\frac{1}{2}}}{\sup_{x \in I, y \in J}{|x-y|}}  \left(\frac{|f|^2_{TV}}{\|f\|^2_{L^2(I)}}+\frac{4}{|I|} \right)^{-1/2} \|f\|_{L^2(I)}.$$
\end{lemma}
\begin{proof} We assume w.l.o.g. $\|f\|_{L^2(I)} = 1$. Since $I$ and $J$ do not overlap, we see that the kernel of the Hilbert transform has constant sign (which sign depends on whether $J$ is to the left or
to the right of $I$). Therefore, since $f$ never changes sign, we have by H\"older and monotonicity,
$$ \| Hf \|_{L^2(J)}  \geq  \frac{1}{|J|^{\frac{1}{2}}} \| Hf \|_{L^1(J)} \geq |J|^{\frac{1}{2}} \frac{1}{\sup_{x \in I, y \in J}{|x-y|}}  \|f\|_{L^1(I)}.$$
Let us now assume that 
$$ \|f\|_{L^1(I)} \leq \varepsilon.$$
Then, there certainly exists a point $x_0$ with $f(x_0) \leq \varepsilon/|I|$ and therefore
$$ \|f\|_{L^{\infty}(I)} \leq \frac{\varepsilon}{|I|} + |f|_{TV}.$$
As a consequence
$$ 1 = \int_{I}{f^2 dx} \leq \|f\|_{L^{\infty}(I)} \int_{I}{|f| dx}$$
and thus
$$ \| f \|_{L^1(I)} \geq \frac{1}{\frac{\varepsilon}{|I|} + |f|_{TV}}$$
from which we derive that
$$ \varepsilon \geq \frac{1}{\frac{\varepsilon}{|I|} + |f|_{TV}}.$$
This shows that $\varepsilon$ cannot be arbitrarily small depending on $|I|$ and $|f|_{TV}$ and simple algebra implies the stated result.
\end{proof}

\textit{Remark.} When $I,J$ are configured as in Case 4, repeating this argument shows that the result of Lemma \ref{polydecay} continues to hold, with the factor $|J|^{1/2}$ replaced by $\frac{1}{2}|J \setminus I|^{1/2}$. 
This argument may also be suitably adapted to show that if $f$ has no root on $I$ and $J^*$ is a subinterval of $J$ that is disjoint from $I$, then
\[ 
 \| Hf \|_{L^2(J)} \geq  \frac{1}{2}|J^* \setminus I |^{1/2}  \frac{1}{\sup\limits_{x \in J^{*}, y \in I} |x-y|} 
  \left(\frac{|\chi_{I \setminus J^*} f|^2_{TV}}{\|f \|^2_{L^2(I \setminus J^*)}} + \frac{4}{|I \setminus J^*|}\right)^{-1/2} \|f\|_{L^2(I \setminus J^*)}.
  \]
This result may be seen as a suitable counterpart to Theorem \ref{thm4}.
\subsection{Proof of Theorem \ref{thm3a}}
\begin{proof}
We now modify the argument used in the first part of the proof of Theorem \ref{thm3} to show Theorem \ref{thm3a}, in which case $f$ is assumed to be in $C_c^{2M+1}(I)$ for some arbitrary $M \geq 1.$
We note that under this strong assumption, which ensures that $f$ and all its first $2M+1$ derivatives vanish at the endpoints of $I$, it follows that $L_I^M f$ also vanishes at the endpoints of $I$. Thus we may apply Lemma \ref{lemma_un} directly to $L_I^M f$.

By  Lemma \ref{lemma_un} and the asymptotics for $\lam_n$,
\begin{align*}
 1 = \|f\|_{L^2(I)}^2 
&= \sum_{n=0}^{N}{\left| \left\langle f, u_n \right\rangle\right|^2}  +  \sum_{n=N+1}^{\infty}{\left| \frac{\lam_n^M}{\lam_n^M}\left\langle f, u_n \right\rangle\right|^2}\\
&= \sum_{n=0}^{N}{\left| \left\langle f, u_n \right\rangle\right|^2}  +  \sum_{n=N+1}^{\infty}{\left| \frac{1}{\lam_n^M}\left\langle f, \lam_n^Mu_n \right\rangle\right|^2}\\
&= \sum_{n=0}^{N}{\left| \left\langle f, u_n \right\rangle\right|^2}  +  \sum_{n=N+1}^{\infty}\frac{1}{\lam_n^{2M}}{\left| \left\langle f, L_I^M u_n \right\rangle\right|^2}\\
&= \sum_{n=0}^{N}{\left| \left\langle f, u_n \right\rangle\right|^2}  +  \sum_{n=N+1}^{\infty}\frac{1}{\lam_n^{2M}}{\left| \left\langle L_I^Mf, u_n \right\rangle\right|^2}\\
&\leq \sum_{n=0}^{N}{\left| \left\langle f, u_n \right\rangle\right|^2} + c^2 \left| L_I^Mf \right|_{\text{TV}}^2\sum_{n=N+1}^{\infty}{\frac{1}{(k_1n^2)^{2M} n^2}} \\
&\leq  \sum_{n=0}^{N}{\left| \left\langle f, u_n \right\rangle\right|^2} + \frac{c^2 k_1^{-2M}  \left| L_I^Mf \right|_{\text{TV}}^2}{N^{4M+1}}.
\end{align*}
We now choose $N$ to be the least integer such that 
\[ N \geq \left( 2c^2k_1^{-2M} \left| L_I^Mf \right|_{TV}^2 \right)^{\frac{1}{4M+1}},\]
so that with this choice, we may set $k_4 = (2c^{2}k_1^{-2M})^{1/(4M+1)}$ to obtain the lower bound
\begin{align*}
 \| H_T f \|_{L^2(J)}^2 &=  \sum_{n=0}^{\infty}{ |\langle f, u_n \rangle |^2 \sigma_n^2 }   \\
&\geq  \sum_{n\leq \lceil k_4 | L_I^M f |_{TV}^{2/(4M+1)}\rceil}{ |\langle f, u_n \rangle |^2 \sigma_n^2 }        \\
&\geq  \left( \sum_{n \leq \lceil k_4| L_I^M f |_{TV}^{2/(4M+1)}\rceil }{ |\langle f, u_n \rangle |^2  } \right) \sigma_{\lceil k_4| L_I^M f |_{TV}^{2/(4M+1)}\rceil}^2 \\
&\geq \frac{1}{2} \exp(-2 k_2 (k_4 | L_I^M f |_{TV}^{2/(4M+1)}+1)) \\
&  \geq  k_5 \exp(-2 k_2 k_4| L_I^M f |_{TV}^{2/(4M+1)})                         
\end{align*}
with $k_5=(1/2) e^{-2 k_2}$. 
We need only note that as $M \maps \infty$, $k_4 \maps k_1^{-1/2}$.

\end{proof}

\subsection{Proof of Corollary \ref{cor-stab}}

\begin{proof}
Let $f_1$ and $f_2$ be elements in $S(\delta,g^{\delta})$. From Theorem \ref{thm3} and $| f_1-f_2 |_{TV} \leq 2 \kappa$, we obtain
\begin{equation*}
\| f_1 - f_2 \|_{L^2(I)} \leq \frac{1}{c_1} e^{c_2 4 \kappa^2/\|f_1 - f_2\|^2_{L^2(I)}} \|H_T (f_1 - f_2) \|_{L^2(J)}.
\end{equation*}
Linearity of $H_T$ and the properties of $S$ then yield
\begin{align*}
\| f_1 - f_2 \|_{L^2(I)} &\leq \frac{1}{c_1} e^{c_2 4 \kappa^2/\|f_1 - f_2\|^2_{L^2(I)}} \|H_T f_1 - H_T f_2 \|_{L^2(J)} \\
& \leq  \frac{1}{c_1} e^{c_2 4 \kappa^2/\|f_1 - f_2\|^2_{L^2(I)}} ( \|H_T f_1-g^{\delta} \|_{L^2(J)} + \|g^{\delta} - H_T f_2 \|_{L^2(J)})  \\
& \leq \frac{1}{c_1} e^{c_2 4 \kappa^2/\|f_1 - f_2\|^2_{L^2(I)}} 2 \delta.
\end{align*}
This gives
\begin{equation*}
\log(\|f_1 - f_2\|_{L^2(I)}) - \frac{c_2 4 \kappa^2}{\|f_1 - f_2\|^2_{L^2(I)}} \leq \log \Big(\frac{2 \delta}{c_1}\Big).
\end{equation*}
A lower bound on the left-hand side of the above inequality can be obtained by observing that $x^2 \log |x| \geq -1/(2e)$ for real-valued $x$. Thus,
\begin{equation*}
-\frac{\frac{1}{2 e}+4 c_2 \kappa^2}{\| f_1-f_2 \|^2_{L^2(I)}} \leq \log \Big(\frac{2 \delta}{c_1}\Big).
\end{equation*}
Hence, if $\delta$ is not too large ($\delta \leq c_1/2$), we can conclude that
\begin{equation}\label{cor-bound}
\| f_1-f_2 \|_{L^2(I)} \leq \sqrt{\frac{\frac{1}{2e}+4 c_2 \kappa^2}{|\log (\frac{2\delta}{c_1})|}}.
\end{equation}

\end{proof}

\section{Proof of Theorem \ref{thm4}.}\label{sub:proof-thm4}

We recall that Theorem \ref{thm4} considers Case 4, with $I \intersect J \neq \emptyset.$
Let $I = (a_2,a_4)$ and $J = (a_1,a_3)$ for $a_1<a_2<a_3<a_4$ and let the subinterval $J^*$ of $J$ be defined as $[a_1+\mu,a_3-\mu]$ for some $\mu>0$ sufficiently small so that $a_1 + \mu < a_2< a_3 - \mu$. We think of $J^*$ as now being fixed for the remainder of the argument.
 For the two accumulation points of the singular values of $H_T$ (the truncated Hilbert transform with overlap), we use the convention $\sigma_n \to 1$ for $n \to -\infty$ and $\sigma_n \to 0$ for $n \to \infty$. The two main ingredients needed for the statement in Theorem \ref{thm4} are the existence of positive constants $B_{\mu}$, $\beta_{\mu}$ and $c$ depending only on $I$, $J$ and $\mu$ such that the following holds for all $n \in \mathbb{N}$: 

\begin{enumerate}
\item $\|  u_n\|_{L^2(I \cap J^*)} \leq B_{\mu} e^{-\beta_{\mu} n},$\label{norm-un-decay}\\
\item $\sup\limits_{x \in I \backslash J^*} |\int_{a_3-\mu}^x u_n(z) dz| \leq \frac{c}{n}.$ \label{int-un-decay}
\end{enumerate}
These properties of the singular functions $u_n$ corresponding to singular values close to zero allow one to estimate the inner products $\langle f,u_n\rangle$. The proof of the first statement can be found in \cite{adk2} for sufficiently large $n$, i.e, $n \geq N_0$ for some $N_0 \in \mathbb{N}$. Since $\|u_n\|_{L^2(I)}=1$ and $N_0$ depends only on $I, J$ and $\mu$, one can easily deduce the existence of constants $B_\mu$, $\beta_\mu$ depending only on $I, J$ and $\mu$ such that (\ref{norm-un-decay}) holds for all $n \in \mathbb{N}$.
Note that we cannot merely apply Lemma \ref{lemma_un} to prove (2), since in the case where $I \intersect J$ is nonempty, the functions $u_n$ behave fundamentally differently at the endpoint $a_3$ of $J$, which lies in $I$.
Thus we prove (2) directly in \S \ref{sec_app_6.2}.

Given any function $g \in W^{1,1}(I)$, we consider the normalization $\tilde{g} = g/\| g\|_{L^2(I)}$, in which case to prove Theorem \ref{thm4} it would suffice to show 
\beq\label{thm4_normalized}
\| H  \tilde{g} \|_{L^2(J)} \geq c_1 \exp (-c_2 |\chi_{I \setminus J^*} \tilde{g}|_{TV}^2),
\eeq 
as long as $\tilde{g}$ satisfies the remaining hypotheses of Theorem \ref{thm4}.

Thus from now on we assume we are working with $f \in W^{1,1}(I)$ and $\|f \|_{L^2(I)}=1$. 
If $f$ vanishes at the boundary of $I \backslash J^*$, we may work directly with $f$. 
Otherwise, if $f$ merely vanishes at least at one point in $I \setminus J^*$, then we may apply a small modification of Lemma \ref{lemma_approx} to approximate $f$ by functions $f_n \in C_c^\infty(I)$ that vanish at the endpoints of $I \setminus J^*$ and such that $\|f_n - f\|_{L^2(I)} \maps 0$ and $|\chi_{I \setminus J^*} f_n|_{TV} \leq 5 |\chi_{I \setminus J^*} f|_{TV}$. Then having proved (\ref{thm4_normalized}) for each $f_n$ we could conclude it holds for $f$, since 
\[ c_1 \exp (-c_2 25|\chi_{I \setminus J^*} f|^2_{TV} ) \leq  \| H f_n\|_{L^2(J)} \leq  \| H f\|_{L^2(J)}  +   \| H (f-f_n)\|_{L^2(J)} ,\]
and $\| H (f-f_n)\|_{L^2(J)} \leq C \|f - f_n\|_{L^2(I)} \maps 0$ as $n \maps \infty$.

Hence, we can assume $\|f\|_{L^2(I)}=1$ and $f$ vanishes at the endpoints of $I \setminus J^*$, so that
\begin{align*}
\left| \int_I f(x) u_n(x) dx \right| &\leq \left| \int_{ I \cap J^*} f(x) u_n(x) dx \right| + \left| \int_{I \backslash J^*} f(x) u_n(x) dx \right| \\
&\leq B_{\mu} e^{-\beta_{\mu} n} + |\chi_{I \backslash J^*} f|_{TV} \sup\limits_{x \in I \backslash J^*} \left| \int_{a_3-\mu}^x u_n(z) dz \right| \\
&\leq B_{\mu} e^{-\beta_{\mu} n} + \frac{c}{n} |\chi_{I \backslash J^*} f|_{TV}. \\
\end{align*}

The remainder of the argument is then similar to the proof of Theorem \ref{thm3}. For any $N \geq 1,$
\begin{align*}
1=\|f\|_{L^2(I)}^{2} &  \leq \sum_{n = -\infty}^{N} \left| \langle f,u_n \rangle \right|^2 + \sum_{n=N+1}^\infty \big( B_{\mu} e^{-\beta_{\mu} n} + \frac{c}{n} |\chi_{I \backslash J^*} f|_{TV} \big)^2 \\
&\leq \sum_{n = -\infty}^{N} \left| \langle f,u_n \rangle \right|^2 +2 B_\mu^2 \sum_{n=N+1}^\infty e^{-2\beta_\mu n} + 2\frac{c^2}{N} |\chi_{I \backslash J^*} f|^2_{TV} \\
&\leq \sum_{n = -\infty}^{N} \left| \langle f,u_n \rangle \right|^2 +2 B_\mu^2 \frac{e^{-2\beta_\mu N}}{e^{2 \beta_\mu}-1} + 2\frac{c^2}{N} |\chi_{I \backslash J^*} f|^2_{TV}.
\end{align*}
Let $\tilde{N}$ be the least integer such that for all $n \geq \tilde{N}$
$$n e^{-2\beta_\mu n} \leq c^2 B_\mu^{-2}\big( e^{2\beta_\mu} - 1 \big)$$
and note that $\tilde{N}$ depends only on $I, J$ and $\mu$. Then, the choice 
\[ N  = \max\{ \tilde{N}, \lceil 4 c^2 (|\chi_{I \backslash J^*} f|_{TV}^2+1)\rceil \} \]
 guarantees that 
the sum $ \sum_{n = -\infty}^{N} \left| \langle f,u_n \rangle \right|^2$ contains at least half of the energy of $f$ and thus
\begin{align*}
\| H_T f \|^2_{L^2(J)} &= \sum_{n=-\infty}^{\infty} \left| \langle f,u_n \rangle \right|^2 \sigma_n^2 \\
&\geq  \sum_{n=-\infty}^{N} \left| \langle f,u_n \rangle \right|^2 \sigma_n^2 \geq \frac{1}{2} \sigma_N^2 \\
& \geq \tilde{k}_0 e^{-k_0 |\chi_{I \backslash J^*}f|^2_{TV}},
\end{align*}
for some constants $k_0, \tilde{k}_0$ depending only on $I$, $J$ and $\mu$.

\section{A remark on generalizations}\label{sec_general}
Let $I, J \subset \mathbb{R}$ be disjoint intervals and let $T:L^2(I) \rightarrow L^2(J)$ be an integral operator of convolution type,
$$ (Tf)(x) = \int_{I}{K(x-y)f(y)dy},$$
for some kernel $K$.
Then we would generically expect an inequality of the type
\begin{equation}\label{inv-statement}
 \| T f\|_{L^2(J)} \geq h\left(\frac{\left| f \right|_{\text{TV}}}{\|f\|_{L^2(I)}}\right)\|f\|_{L^2(I)}
\end{equation}
to hold true,
for some positive function $h:\mathbb{R}_{+} \rightarrow \mathbb{R}_{+}$. The purpose
of this section is to show how to construct examples where the function $h$ depends very strongly
on very fine properties of the kernel $K$. 

\subsection{Our example.} 
For reasons of clarity, we set $I = [0,1]$ and take $K:\mathbb{R} \rightarrow 
\mathbb{R}$ to be a $1$--periodic smooth function. We define the integral operator $T:L^2([0,1]) \rightarrow L^{\infty}(\mathbb{R})$ by
$$ (Tf)(x) = \int_{0}^{1}{K(x-y)f(y)dy}.$$
The function $Tf$ is also periodic with period 1. We will not specify the interval $J$ because it will be irrelevant. 
The main idea is that we can identify
$$Tf = K *f$$
 with a function on the torus $\mathbb{T}$ (normalized to have length 1). Expressing everything in terms of Fourier series yields
\[ \sum_n \widehat{Tf}(n) = \sum_n \hat{K}(n) \hat{f}(n).\]
We now see that if the Fourier coefficients of $K$ and $f$ are supported on disjoint sets of frequencies, then we immediately get $Tf = 0$.
Put differently, the only way to ensure that $Tf \neq 0$ for every $f \neq 0$ is to ensure that $K$ has no vanishing Fourier coefficients.
\begin{lemma}[Folklore]\label{lemma_folklore}
Let $K \in L^2(\mathbb{T})$. Then the span of $\left\{K(x - a): a \in \mathbb{T}\right\}$ is dense in $L^2(\mathbb{T})$ if and only if
$$ \forall~n \in \mathbb{Z}, \qquad \hat{K}(n) \neq 0.$$
\end{lemma}
\begin{proof}
 One direction is easy: if $\hat{K}(n) = 0$ for some $n \in \mathbb{Z}$, then $e^{i n x}$ serves as a counterexample. As for the other direction, suppose $g \in L^2(\mathbb{T})$
is orthogonal to all translations of $K$. Then, for any $t \in \mathbb{T}$, by Parseval
$$ 0 = \int_{\mathbb{T}}{K(x)g(x-t)dx} = \sum_{n \in \mathbb{Z}}{\hat{K}(n)\overline{e^{int}\hat{g}(n)}} = \sum_{n \in \mathbb{Z}}{\hat{K}(-n)\overline{\hat{g}(-n)}e^{int}}.$$ 
Since $t$ was arbitrary, this means that the Fourier series
$$ \sum_{n \in \mathbb{Z}}{\hat{K}(-n)\overline{\hat{g}(-n)}e^{int}} $$
vanishes identically and since for all $n,$ $\hat{K}(-n) \neq 0$, this implies that $g = 0.$
\end{proof}
Having established this lemma, the proof of an estimate of the type
$$ \| T f\|_{L^2(J)} \geq h\left(\frac{\left| f \right|_{\text{TV}}}{\|f\|_{L^2(I)}}\right)\|f\|_{L^2(I)},$$
for some positive-valued function $h$ is easy. If we take a minimizing sequence $f_{n_k} \in BV(I)$, Helly's compactness theorem implies the existence of a convergent subsequence $f_{n_k} \rightarrow f$ with a pointwise limit $f \in BV(I)$. 
Assuming that $K \in L^2(\mathbb{T})$ has $\hat{K}(n) \neq 0$ for all $n$, Lemma \ref{lemma_folklore} implies that the translates of $K$ are dense in $L^2(I)$. Then, however, it is impossible for the operator $T$ to map $f$ to 0 and this proves the statement.

\subsection{Conclusion.} In order for an inequality of
the type 
$$ \| T f\|_{L^2(J)} \geq h\left(\frac{\left| f \right|_{\text{TV}}}{\|f\|_{L^2(I)}}\right)\|f\|_{L^2(I)}$$
to hold true at all, fine properties of the Fourier coefficients of the kernel play a crucial role. Furthermore, even assuming such an inequality to be true, the quantitative rate of decay of $h$ will directly depend on
the speed with which the Fourier coefficients decay to 0: it is thus possible to construct explicit examples
of kernels $K$ for which the associated function $h$ decays faster than any arbitrary given function.
These are very serious obstructions for any generalized theory of bounding truncated integral operators
from below if one were to hope that such a theory could be stated in `rough' terms (i.e. smoothness
of the function, $L^p$--norms of the kernel $K$ and its derivatives). In the example above, bounding
Fourier coefficients $\hat{K}(n)$ from below seems unavoidable. 

\section{Appendix}

\subsection{Proof of Lemma \ref{lemma_approx}}\label{sec_lemma1}

\begin{proof}
A function $f \in BV(I)$ can be approximated by smooth functions in the following way 
\cite[Section 3.1]{ambrosio}: There exists a sequence $\{ f_n \} \in C^{\infty}(I) \cap BV(I)$ such that
\begin{align}
\| f_n - f \|_{L^1(I)} \to 0, \label{l1-conv}\\
|f_n|_{TV} \to |f|_{TV}. \label{smooth-appr-tv}
\end{align}

We are seeking an approximation by smooth functions that vanish at the boundary of $I$. Since $C_c^{\infty}(I) \subset BV(I)$ is dense in $L^1(I)$, one can also find a sequence $f_n \in C_c^{\infty}(I),$ that satisfies (\ref{l1-conv}). Now instead of \eqref{smooth-appr-tv}, we use the fact that $f(x_0)=0$ for some $x_0 \in I$ 
to note that $\|f \|_{L^\infty(I)} \leq |f|_{TV}$, and so instead of (\ref{smooth-appr-tv}) we now have 
\begin{equation}
\label{ff_TV}
 \left| f_n \right|_{TV} \leq \left| f \right|_{TV} + 2\|f\|_{L^\infty(I)} \leq 3 | f |_{TV}. 
\end{equation}
Finally, $L^2$--convergence can be obtained as follows by noting that $\{f_n\}$ is uniformly bounded.
Indeed, suppose there exists a subsequence $\{f_{n_k} \}$ such that $\|f_{n_k} \|_{L^{\infty}(I)} > 3 \left| f \right|_{TV} + \varepsilon$ for some small $\varepsilon >0.$ Then, each $f_{n_k}$ does not change sign. For supposing that it did, we would have 
$$ |f_{n_k}|_{TV} \geq  \|f_{n_k} \|_{L^\infty(I)} > 3| f|_{TV} + \ep,$$
which contradicts (\ref{ff_TV}).

 Thus we may assume w.l.o.g. $f_{n_k} \geq0,$ in which case we see that for each $x \in I$, $$\|f_{n_k}\|_{L^\infty(I)} - f_{n_k}(x) \leq |f_{n_k}|_{TV}.$$ 

 This yields

$$0 <\ep< \| f_{n_k} \|_{L^{\infty}(I)} - 3 \left| f \right|_{TV}  \leq \| f_{n_k} \|_{L^{\infty}(I)} -  \left| f \right|_{TV}\leq f_{n_k}(x), \quad \forall x \in I.$$

Furthermore,
$$\int_I ( \| f_{n_k} \|_{L^{\infty}(I)} - 3 \left| f \right|_{TV}) dx \leq \|f_{n_k} \|_{L^1(I)} \leq 2 \| f\|_{L^1(I)},$$
which results in the uniform bound $\| f_{n_k} \|_{L^{\infty}(I)} \leq 2 \|f\|_{L^1(I)}/|I|+ 3 \left| f \right|_{TV}.$
Since $L^1$--convergence in \eqref{l1-conv} implies the existence of a subsequence $\{ f_{n_k}\}$ of $\{ f_n\}$ such that $f_{n_k} \xrightarrow{\text{pw}} f$ almost everywhere, the dominated convergence theorem results in $$\| f_{n_k} - f \|_{L^2(I)} \to 0.$$
\end{proof}

\subsection{Proof of Lemma \ref{int-un}}\label{sec_int-un}

\begin{proof} Here we will prove the statement
\begin{equation*}
\forall~x\in (a_3, a_4): \qquad  \left| \int_{a_3}^{x}{u_n(z)dz}\right| \leq \frac{c}{n},
\end{equation*}
where $u_n$ is the $n$-th eigenfunction of $L_{I}$ with associated eigenvalue $\lambda_n$. 
We recall we are in the case where $I$ and $J$ are disjoint, with $I=(a_3,a_4)$. We choose $N_0 \in \mathbb{N}$ (depending only on $I$ and $J$) such that the asymptotic form of $u_n$ in \cite{kat3} is valid for all $n \geq N_0$. We first show the result for $n \geq N_0$. For this, we note that on $(a_3,a_4)$ and away from the points $a_3$ and $a_4$, the function $u_n$ can be approximated by the Wentzel--Kramers--Brillouin (WKB) solution. More precisely, defining $\varepsilon = \varepsilon_n := 1/\sqrt{\lambda_n}$, it is true that for any sufficiently small $\delta >0$, the representation of $u_n$ in the form
\begin{multline*}
u_n(z) = \frac{K}{(-P(z))^{1/4}} \left[ \cos \Big( \frac{1}{\varepsilon} \int_{a_3}^z \frac{dt}{\sqrt{-P(t)}} -\frac{\pi}{4} \Big) \cdot (1+ \mathcal{O}(\varepsilon^{1/2-\delta})) \right. \\
	\left. +   \sin \Big( \frac{1}{\varepsilon} \int_{a_3}^z \frac{dt}{\sqrt{-P(t)}} -\frac{\pi}{4} \Big) \cdot \mathcal{O}(\varepsilon^{1/2-\delta}) \right] \nonumber
\end{multline*}
is valid for $z \in [a_3 + \mathcal{O}(\varepsilon^{1+2\delta}), a_4 - \mathcal{O}(\varepsilon^{1+2\delta})]$ and some positive constant $K$ depending only on $a_1,a_2,a_3,a_4$. Having this, we start by estimating 
\begin{equation*}
\left| \int_{a_3+ \mathcal{O}(\varepsilon^{1+2\delta})}^{x}{u_n(z)dz}\right|
\end{equation*}
for $x \in [a_3 + \mathcal{O}(\varepsilon^{1+2\delta}), a_4 - \mathcal{O}(\varepsilon^{1+2\delta})]$. We do this by first introducing $\tilde{u}_n(z) = (-P(z))^{-1/4} u_n(z)$, for which
\begin{align*}
\int_{a_3+ \mathcal{O}(\varepsilon^{1+2\delta})}^{x}{\tilde{u}_n(z)dz} =&K \varepsilon \Big[ \sin \Big( \frac{1}{\varepsilon} \int_{a_3}^x \frac{dt}{\sqrt{-P(t)}} -\frac{\pi}{4} \Big) \\
&\quad - \sin \Big( \frac{1}{\varepsilon} \int_{a_3}^{a_3+ \mathcal{O}(\varepsilon^{1+2\delta})} \frac{dt}{\sqrt{-P(t)}} -\frac{\pi}{4} \Big) \Big] \cdot (1+ \mathcal{O}(\varepsilon^{1/2-\delta})) \\
&- K \varepsilon \Big[ \cos \Big( \frac{1}{\varepsilon} \int_{a_3}^x \frac{dt}{\sqrt{-P(t)}} -\frac{\pi}{4} \Big) \\
&\quad - \cos \Big( \frac{1}{\varepsilon} \int_{a_3}^{a_3+ \mathcal{O}(\varepsilon^{1+2\delta})} \frac{dt}{\sqrt{-P(t)}} -\frac{\pi}{4} \Big) \Big] \cdot \mathcal{O}(\varepsilon^{1/2-\delta})
\end{align*}
and hence
$$\left| \int_{a_3+ \mathcal{O}(\varepsilon^{1+2\delta})}^{x}{\tilde{u}_n(z)dz}\right|  \leq K \varepsilon (1+ \mathcal{O}(\varepsilon^{1/2-\delta})).$$
It is known from the asymptotics derived in \cite{kat3} that 
\beq\label{ep_n}
\varepsilon =~ \varepsilon_n = \frac{2}{K^2 n \pi} + \mathcal{O}(n^{-1/2+\delta}).
\eeq
Thus, there exists a constant $\tilde{c}_1$ depending only on $a_1, a_2,a_3, a_4$ such that
\begin{equation*}
\left| \int_{a_3+ \mathcal{O}(\varepsilon^{1+2\delta})}^{x}{\tilde{u}_n(z)dz}\right| \leq \frac{\tilde{c}_1}{n}.
\end{equation*}
We can use this together with integration by parts to find an upper bound on the above expression with $\tilde{u}_n$ replaced by $u_n$:
\begin{align*}
\int_{a_3+ \mathcal{O}(\varepsilon^{1+2\delta})}^{x} u_n(z) dz &= \int_{a_3+ \mathcal{O}(\varepsilon^{1+2\delta})}^{x} (-P(z))^{1/4} \tilde{u}_n(z) dz \\
&= - \int_{a_3+ \mathcal{O}(\varepsilon^{1+2\delta})}^{x} \frac{d}{dz}  (-P(z))^{1/4} \int_{a_3+ \mathcal{O}(\varepsilon^{1+2\delta})}^{z}  \tilde{u}_n(t) dt dz + \\
&+ \left. \Big( (-P(z))^{1/4} \int_{a_3+ \mathcal{O}(\varepsilon^{1+2\delta})}^{z}  \tilde{u}_n(t) dt \Big) \right|_{a_3+ \mathcal{O}(\varepsilon^{1+2\delta})}^x.
\end{align*}
This gives
\begin{align*}
&\left| \int_{a_3+ \mathcal{O}(\varepsilon^{1+2\delta})}^{x} u_n(z) dz \right| \\
&\ \ \leq \sup\limits_{z \in [a_3+\mathcal{O}(\varepsilon^{1+2\delta}),x]} \left|  \int_{a_3+ \mathcal{O}(\varepsilon^{1+2\delta})}^{z}  \tilde{u}_n(t) dt \right| \cdot \int_{a_3+\mathcal{O}(\varepsilon^{1+2\delta})}^x \left| \frac{d}{dz} (-P(z))^{1/4} \right| dz \\
&\quad \quad + \left| P(x) \right|^{1/4} \cdot \left| \int_{a_3+\mathcal{O}(\varepsilon^{1+2\delta})}^x \tilde{u}_n(t) dt \right| \leq \frac{c_1}{n},
\end{align*}
for some constant $c_1$ that depends only on the points $a_i$. Here we have used that $\frac{d}{dz} (-P(z))^{1/4}$ changes sign exactly once within $(a_3,a_4)$ and hence
$$\int_{a_3+\mo(\varepsilon^{1+2\delta})}^x \left| \frac{d}{dz} (-P(z))^{1/4} \right| dz \leq \sup_{a_3 \leq x \leq a_4}4(-P(x))^{1/4}.$$

What remains to be shown is the estimate for the contributions close to the points $a_3$ and $a_4$. Since (by the definition of the operator $L_I$) the asymptotic behavior of $u_n$ at $a_4$ is identical to its behavior at $a_3$, it suffices to find an upper bound on $$\left| \int_{a_3}^{x} u_n(z) dz \right|, \quad x \in (a_3,a_3+\mathcal{O}(\varepsilon^{1+2\delta})]. $$
On this interval, $(a_3,a_3+\mathcal{O}(\varepsilon^{1+2\delta})]$, the eigenfunctions $u_n$ can be approximated by the Bessel function $J_0$. (This approximation is specific to the case where $I$ and $J$ are disjoint.) For this, we define the variable $t=(a_3-z)/(\varepsilon^2 P'(a_3))$. Then, the asymptotic behavior of $u_n$ has been found to be
\begin{equation*}
u_n(z) = \begin{cases} b_3 [J_0(2\sqrt{t})+\mathcal{O}(\varepsilon^{1-2\delta/3})], & \mbox{for } t \in [0,1) \\ b_3 [J_0(2\sqrt{t})+t^{-1/4} \mathcal{O}(\varepsilon^{1-2\delta/3})], & \mbox{for } t \in [1,\mathcal{O}(\varepsilon^{2\delta-1})] \end{cases}
\end{equation*}
with a constant $b_3=\mathcal{O}(\varepsilon^{-1/2})$. A change of variables $dx = - \varepsilon^2 P'(a_3) dt$ and $t(x)=\frac{a_3-x}{\e^2 P'(a_3)} = \mo(\e^{2\delta-1})$ then yield
\begin{align*}
\int_{a_3}^{x} u_n(z) dz =& b_3 \cdot \Big\{ \int_0^1 \big[ J_0(2\sqrt{t}) + \mathcal{O}(\varepsilon^{1-2\delta/3}) \big] \varepsilon^2 (-P'(a_3)) dt \\
&+ \int_1^{\mathcal{O}(\varepsilon^{2\delta-1})} \big[ J_0(2\sqrt{t}) + t^{-1/4} \mathcal{O}(\varepsilon^{1-2\delta/3}) \big] \varepsilon^2 (-P'(a_3)) dt \Big\} \\ 
=& \mo(\e^{3/2}) \cdot \Big\{ \int_0^1 \big[ J_0(2\sqrt{t}) + \mathcal{O}(\varepsilon^{1-2\delta/3}) \big] dt \\
&+ \int_1^{\mathcal{O}(\varepsilon^{2\delta-1})} \big[ J_0(2\sqrt{t}) + t^{-1/4} \mathcal{O}(\varepsilon^{1-2\delta/3}) \big]  dt \Big\}.
\end{align*}
The first integral in the above sum is bounded, thus
\begin{equation*}
\int_{a_3}^{x} u_n(z) dz = \mo(\e^{3/2}) + \mo(\e^{3/2}) \cdot  \int_1^{\mathcal{O}(\varepsilon^{2\delta-1})} \big[ J_0(2\sqrt{t}) + t^{-1/4} \mathcal{O}(\varepsilon^{1-2\delta/3}) \big]  dt.
\end{equation*}
To find an upper bound on the remaining integral, we first estimate it by
\begin{align}
&\left|  \int_1^{\mathcal{O}(\varepsilon^{2\delta-1})} \big[ J_0(2\sqrt{t}) + t^{-1/4} \mathcal{O}(\varepsilon^{1-2\delta/3}) \big]  dt \right| \nonumber \\
&\qquad \qquad = \left| \int_1^{\mathcal{O}(\varepsilon^{2\delta-1})}  J_0(2\sqrt{t}) dt + t^{3/4} \mathcal{O}(\varepsilon^{1-2\delta/3}) \Big\vert_1^{\mo(\e^{2\delta-1})}  \right| \nonumber \\
& \qquad \qquad \leq \left| \int_1^{\mathcal{O}(\varepsilon^{2\delta-1})} J_0(2\sqrt{t}) dt \right| + \mo(\e^{1/4+5\delta/6}). \label{int-J0-1}
\end{align}
Next, we make use of the asymptotic form of $J_0$ for $t \to \infty$:
\begin{equation}\label{int-J0-2}
J_0(2\sqrt{t}) = \frac{1}{\sqrt{\pi} t^{1/4}} \big[ \cos(2\sqrt{t}-\frac{\pi}{4}) + \mo(t^{-1/2})\big].
\end{equation}
For some fixed, sufficiently large $T$, we can write
\begin{align}
&\left| \int_1^{\mathcal{O}(\varepsilon^{2\delta-1})} J_0(2\sqrt{t}) dt \right| \nonumber \\
&\qquad \leq \left| \int_1^T J_0(2\sqrt{t}) dt \right|+ \left| \int_T^{\mathcal{O}(\varepsilon^{2\delta-1})} \big[ \frac{1}{\sqrt{\pi} t^{1/4}} \cos(2\sqrt{t} - \frac{\pi}{4}) + \mo(t^{-3/4}) \big] dt \right| \nonumber \\
&\qquad \leq \tilde{\tilde{c}}_2 + \left| \frac{1}{\sqrt{2 \pi}}t^{1/4} \big[ - \cos(2\sqrt{t}) + \sin(2\sqrt{t}) \big] \Big\vert_T^{\mathcal{O}(\varepsilon^{2\delta-1})} \right| + \mo(\e^{-1/4+\delta/2}) \nonumber \\
& \qquad \leq \tilde{c}_2 + \mo(\e^{-1/4+\delta/2}), \label{int-J0-3}
\end{align}
for some constants $\tilde{c}_2, \tilde{\tilde{c}}_2$, where the second inequality is obtained by explicit evaluation in Mathematica.

This yields
$$\left| \int_{a_3}^{x} u_n(z) dz \right| \leq \mo(\e^{3/2})+ \mo(\e^{3/2} \cdot \e^{-1/4+\delta/2}) + \mo(\e^{3/2} \cdot \e^{1/4+5\delta/6}) = \mo(\e^{5/4+\delta/2}),$$ 
where we have recalled from (\ref{ep_n}) that for sufficiently large $n$, $\ep = \ep_n < 1$.
Consequently, this integral decays at least as fast as $\mo(n^{-1})$, and we may conclude that there exists a constant $c_2$ such that
\begin{equation}\label{decay-J0-part}
\left| \int_{a_3}^{x} u_n(z) dz \right| \leq \frac{c_2}{n}, \quad x \in (a_3,a_3+\mo(\e^{1+2\delta})].
\end{equation}
Altogether, this implies the existence of a constant $\tilde{c}$ depending only on $I$ and $J$ for which
$$\left| \int_{a_3}^x u_n(z) dz \right| \leq \frac{\tilde{c}}{n}, \quad x \in (a_3,a_4),$$
given that $n \geq N_0$. Trivially, however, the following upper bound can be derived for $n < N_0$ by noting that $\|u_n\|_{L^2(I)}=1$:
$$\left| \int_{a_3}^x u_n(z) dz \right| \leq \int_{a_3}^{a_4}  \left|u_n(z) \right| dz \leq C \leq \frac{C N_0}{n}$$
for $C = (a_4-a_3)^{1/2}$. Thus, with the choice $c = \max \{ \tilde{c},C N_0 \}$ the assertion holds for all $n \in \mathbb{N}$. 

\end{proof}

\subsection{Proof of Relation \eqref{int-un-decay} in \S \ref{sub:proof-thm4}}\label{sec_app_6.2}
Here, we recall that we are considering Case 4, with $I=(a_2,a_4)$ and $J=(a_1,a_3)$ overlapping intervals with $a_1< a_2<a_3<a_4$, and $\mu>0$ is fixed so that $a_2< a_3 - \mu$. We will expand the above argument for bounding integrals of $u_n$ to this case with overlap, As a consequence of the fact that $\sigma_n \to 0$ (or equivalently $\lambda_n \to +\infty$), we will prove that for all $x \in [a_3-\mu,a_4]$,
$$\left| \int_{a_3-\mu}^x u_n(z) dz \right| \leq \frac{c}{n}.$$
As before, we define $\ep=\varepsilon_n = 1/\sqrt{\lambda_n}$ and omit the index. For sufficiently large $n$, the WKB approximation is valid on $[a_3-\mu, a_3-\mo(\e^{1+2\delta})]$ and is given by
$$u_n(z) = \frac{K}{(P(z))^{1/4}} e^{-\frac{1}{\e} \int_z^{a_3} \frac{dt}{\sqrt{P(t)}}} \cdot \big( 1+ \mo(\e^{1/2-\delta})\big),$$ for the same constant $K$ as in \S \ref{sec_int-un}. With this pointwise decay of $u_n$ that is exponential in $n$, one easily sees that for $x \in [a_3-\mu, a_3-\mo(\e^{1+2\delta})]$ the integral $|\int_{a_3-\mu}^x u_n(z) dz|$ decays faster than $\mo(1/n)$.

Next, we consider $x \in [a_3-\mo(\e^{1+2\delta}),a_4]$. We distinguish three different cases into which we can split the integrals as follows:
for $ x \in [a_3-\mo(\e^{1+2\delta}),a_3],$
\[\left| \int_{a_3-\mu}^x u_n(z) dz \right|  \leq \left| \int_{a_3-\mu}^{a_3-\mo(\e^{1+2\delta})} u_n(z) dz \right| + \left| \int_{a_3-\mo(\e^{1+2\delta})}^{x} u_n(z) dz \right|;\]
for $ x \in [a_3,a_3+\mo(\e^{1+2\delta})],$
\[
\left| \int_{a_3-\mu}^x u_n(z) dz \right|  \leq \left| \int_{a_3-\mu}^{a_3 - \mo(\ep^{1+2\del})} u_n(z) dz \right| + \left| \int_{a_3 - \mo(\ep^{1+2\del})}^{a_3} u_n(z) dz \right| + \left| \int_{a_3}^{x} u_n(z) dz \right|;\]
and for $x \in [a_3+\mo(\e^{1+2\delta}),a_4],$
\begin{multline*}
\left| \int_{a_3-\mu}^x u_n(z) dz \right|  \leq \left| \int_{a_3-\mu}^{a_3-\mo(\e^{1+2\delta})} u_n(z) dz \right|  \\
 + 2 \left| \int_{a_3}^{a_3+\mo(\e^{1+2\delta})} u_n(z) dz \right| 
+ \left| \int_{a_3+\mo(\e^{1+2\delta})}^{x} u_n(z) dz \right|.
\end{multline*}
The last inequality relies on a property of the singular functions $u_n$ that is referred to as \textit{transmission conditions} (see \cite{reema1} for details). Roughly, it states that the parts of $u_n$ on regions of size $\mo(\e^{1+2\delta})$ from the left and from the right of the point of singularity $a_3$ are the same as they approach the limit to $a_3$.

If we let $A$ represent an integral over an interval at least $\mo(\ep^{1+2\del})$ away from the left of $a_3$, $B$ represent an integral within an $\mo(\ep^{1+2\del})$ neighborhood to the left or right of $a_3$ (the transmission conditions ensure the left-hand and right-hand cases are equivalent), and $C$ represent an integral over an interval at least $\mo(\ep^{1+2\del})$ away from the right of $a_3$, we see that the right hand sides of the above three inequalities take the form $A + B$, $A + B + B$, and $A + 2B + C$, respectively.

Integrals of the form $A$ decay at least to order $\mo(1/n)$, as remarked above. Integrals of the form $C$ may be shown to decay to order $\mo(1/n)$ by the argument of Section \ref{sec_int-un}, since the behavior of $u_n$ away from $a_3$ is independent of whether $I$ and $J$ intersect.

What remains is to treat the case of integrals of the form $B$, that is, to show that for $x \in (a_3,a_3+\mo(\e^{1+2\delta})]$,
$$ \left| \int_{a_3}^{x} u_n(z) dz \right| \leq \frac{\tilde{c}}{n}, $$
for some $\tilde{c} >0$. For this, we can proceed in a similar fashion as in \S \ref{sec_int-un}, with the key change that where in \S \ref{sec_int-un} we used an approximation of $u_n$ by the Bessel function $J_0$ on this region, now, in the case of overlapping intervals $I$ and $J$, $u_n$ is no longer a bounded function close to $a_3$, but can be approximated by a linear combination of the Bessel functions $J_0$ and $Y_0$.
 More precisely, substituting $t=(a_3-z)/(\varepsilon^2 P'(a_3))$ yields,
\begin{equation*}
u_n(z) = \begin{cases} b_3 [J_0(2\sqrt{t})+\mathcal{O}(\varepsilon^{1-2\delta/3})] + c_3 [Y_0(2\sqrt{t})+\mathcal{O}(\varepsilon^{3/2-\delta/3})], t \in [0,1) \\ 
b_3 [J_0(2\sqrt{t})+t^{-1/4} \mathcal{O}(\varepsilon^{1-2\delta/3})]\\
\quad + c_3 [Y_0(2\sqrt{t})+t^{-1/4} \mathcal{O}(\varepsilon^{1-2\delta/3})], t \in [1,\mathcal{O}(\varepsilon^{2\delta-1})] \end{cases}
\end{equation*}
with constants $b_3=\mathcal{O}(\varepsilon^{-\delta})$ and $c_3 = \mo(\e^{-1/2})$. As before, with a change of variables $dx = - \varepsilon^2 P'(a_3) dt$, we obtain
\begin{align*}
\int_{a_3}^{x} u_n(z) dz =& b_3 \varepsilon^2 (-P'(a_3)) \cdot \Big\{ \int_0^1 \big[ J_0(2\sqrt{t}) + \mathcal{O}(\varepsilon^{1-2\delta/3}) \big]  dt \\
&+ \int_1^{\mathcal{O}(\varepsilon^{2\delta-1})} \big[ J_0(2\sqrt{t}) + t^{-1/4} \mathcal{O}(\varepsilon^{1-2\delta/3}) \big] dt \Big\} \\ 
&+c_3 \varepsilon^2 (-P'(a_3)) \cdot \Big\{ \int_0^1 \big[ Y_0(2\sqrt{t}) + \mathcal{O}(\varepsilon^{3/2-\delta/3}) \big]  dt \\
&+ \int_1^{\mathcal{O}(\varepsilon^{2\delta-1})} \big[ Y_0(2\sqrt{t}) + t^{-1/4} \mathcal{O}(\varepsilon^{1-2\delta/3}) \big] dt \Big\}.
\end{align*}
Using the results from the proof in \S \ref{sec_int-un} for the terms involving $J_0$, this simplifies to
\begin{align*}
\left| \int_{a_3}^{a_3+\mathcal{O}(\varepsilon^{1+2\delta})} u_n(z) dz \right|  \leq  \mo(\e^{2-\delta}) & + \mo(\e^{3/2}) \cdot \Big\{ \left| \int_0^1 \big[ Y_0(2\sqrt{t}) + \mathcal{O}(\varepsilon^{3/2-\delta/3}) \big] dt \right| \\
& + \left| \int_1^{\mathcal{O}(\varepsilon^{2\delta-1})} \big[ Y_0(2\sqrt{t}) + t^{-1/4} \mathcal{O}(\varepsilon^{1-2\delta/3}) \big] dt \right| \Big\}.
\end{align*}
The first integral on the right-hand side of the above is bounded, since for small arguments $z$, $Y_0(z) \sim \frac{2}{\pi} \ln (z)$. For the second integral, the same argument as for $J_0$ in \eqref{int-J0-1}--\eqref{int-J0-3} holds, but upon replacing the asymptotic form \eqref{int-J0-2} by 
$$Y_0(2\sqrt{t}) = \frac{1}{\sqrt{\pi} t^{1/4}}  \big[ \sin(2\sqrt{t}-\frac{\pi}{4}) + \mo(t^{-1/2})\big].$$
This then allows us to state that 
$$\left| \int_{a_3}^x u_n(z) dz \right| \leq \frac{\tilde{c}}{n}, \quad \forall x \in [a_3,a_3+\mo(\e^{1+2\delta})]$$
and consequently, that for all $x \in [a_3-\mu, a_4]$, 
$$\left| \int_{a_3-\mu}^x u_n(z) dz \right| \leq \frac{c}{n}.$$

\bigskip
\textbf{Acknowledgment.} We are grateful to Angkana R\"uland, Ingrid Daubechies, Michel Defrise, Herbert Koch and Christoph Thiele for valuable comments. The first author was supported by a fellowship of the Research Foundation Flanders (FWO), the second author is supported in part by NSF grant DMS-1402121 and the third author was supported by a Hausdorff scholarship of the
Bonn International Graduate School and the SFB Project 1060 of the DFG.

\end{document}